\documentclass[sn-mathphys-ay]{sn-jnl}% Math and Physical Sciences Author Year Reference Style
\usepackage{graphicx}%
\usepackage{multirow}%
\usepackage{amsmath,amssymb,amsfonts}%
\usepackage{amsthm}%
\usepackage{mathrsfs}%
\usepackage[title]{appendix}%
\usepackage{xcolor}%
\usepackage{textcomp}%
\usepackage{manyfoot}%
\usepackage{booktabs}%
\usepackage{algorithm}%
\usepackage{algorithmicx}%
\usepackage{algpseudocode}%
\usepackage{listings}%
\numberwithin{equation}{section}
\newtheorem{theorem}{Theorem}[section]
\newtheorem{lemma}[theorem]{Lemma}
\newtheorem{corollary}[theorem]{Corollary}

\newcommand{\ep}{\mbox{\scriptsize\textcircled{$\dagger$}}}
\newcommand{\core}{\mbox{\scriptsize\textcircled{\#}}}
\newcommand{\rg}{{\mathscr{R}}}
\usepackage{mathrsfs}
\newcommand{\nl}{{\mathscr{N}}}
\newtheorem{example}{Example}%
\newtheorem{definition}{Definition}%

\begin{document}

\title[Dual Core-EP Generalized Inverse and Decomposition]{Dual Core-EP Generalized Inverse and Decomposition}

%%=============================================================%%
%% GivenName	-> \fnm{Joergen W.}
%% Particle	-> \spfx{van der} -> surname prefix
%% FamilyName	-> \sur{Ploeg}
%% Suffix	-> \sfx{IV}
%% \author*[1,2]{\fnm{Joergen W.} \spfx{van der} \sur{Ploeg} 
%%  \sfx{IV}}\email{iauthor@gmail.com}
%%=============================================================%%

\author[1]{\fnm{Bibekananda} \sur{Sitha}}\email{p20190066@goa.bits-pilani.ac.in}

\author[1]{\fnm{Jajati Keshari} \sur{Sahoo}}\email{jksahoo@goa.bits-pilani.ac.in}
\equalcont{These authors contributed equally to this work.}

\author*[3]{\fnm{N\'estor} \sur{Thome}}\email{njthome@mat.upv.es}
\equalcont{These authors contributed equally to this work.}

\affil[1]{\orgdiv{Department of Mathematics Goa, India}, \orgname{BITS Pilani, K.K. Birla Goa Campus}, \orgaddress{\street{Zuarinagar}, \city{South Goa}, \postcode{403726}, \state{Goa}, \country{India}}}

\affil[2]{\orgdiv{Instituto Universitario de Matem\'atica Multidisciplinar}, \orgname{Universitat Polit\`ecnica de Val\`encia}, \orgaddress{\street{Camino de Vera}, \city{Valencia}, \postcode{46022}, \state{Valencia}, \country{Spain}}}

%%==================================%%
%% Sample for unstructured abstract %%
%%==================================%%

\abstract{In this work, we introduce a new type of generalized inverse called dual core-EP generalized inverse (in short DCEPGI) for dual square matrices. We analyze the existence and uniqueness of the DCEPGI inverse and its compact formula using dual Drazin and dual MP inverse. Moreover, some characterizations using core-EP decomposition are obtained. We present a new dual matrix decomposition named the dual core-EP decomposition for square dual matrices. In addition, some relationships with other dual generalized inverses are established. As an application, solutions to some inconsistent system of linear dual equations are derived.}


%%\pacs[JEL Classification]{D8, H51}

%%\pacs[MSC Classification]{35A01, 65L10, 65L12, 65L20, 65L70}

\maketitle

\section{Introduction}\label{sec1}

A dual number is represented by $\widehat{a}=a+\epsilon b$ \cite{dualno}, where $a$ and $b$ are real numbers, and $\epsilon$ is dual unit with $\epsilon\neq 0,~\epsilon ^2=0,~\epsilon 1=1 \epsilon=\epsilon,\text{ and }\epsilon 0=0 \epsilon=0$. The use of dual numbers and their algebraic properties have emerged as robust and convenient tools analyzing mechanical systems. Their popularity has been increasing over the past three decades, particularly in engineering domains, including kinematic analysis \cite{angeles1998application,angeles2012dual}, screw motion \cite{jin2010application}, robotics \cite{gu1987dual,heibeta1986homogeneous}, and  rigid body motion analysis \cite{dualno}. It is established that the collection of dual numbers form a ring structure. Further extensions of dual numbers are dual vector and dual matrices.

In the area of kinematic analysis and machine synthesis, the dual generalized inverses of dual matrices are often employed to solve a variety of problems. It is important to mention that, in contrast to real matrices, dual matrices not always have dual generalized inverses. Due to this reason, the interest among researchers to investigate the existence of dual generalized inverses and to develop efficient algorithms for their computation in cases where they exist has been increasing as well as applications of dual generalized inverses.

Falco et al. addressed the existence criterion for several kinds of dual generalized inverses in \cite{de2018generalized}. Additionally, to demonstrate the usefulness and adaptability of dual generalized inverses, discussions of solutions to various kinematic problems were examined. Udwadia discussed the existence of several types of dual generalized inverses for dual matrices in \cite{udwadia2021does}. A few new fundamental results regarding the sufficient and necessary criteria of different types of dual generalized inverses have been derived. Some innovative and computationally efficient formulas for computing the MPDGI were introduced in \cite{pennestri2018moore}. Udwadia et al. \cite{udwadia2020all} answered whether every dual matrix has dual MP generalized inverses and illustrated that there exist many dual matrices that do not have dual MP inverse. Udwadia \cite{udwadia2021} investigated the characteristics of the DMPGI and applied them to solve linear dual equation systems. Wang \cite{wang2021characterizations} gave a formula for determining the DMPGI and necessary and sufficient criteria for a dual matrix to have the DMPGI. Furthermore, Wang and Gao \cite{wang2023dual} investigated the dual core generalized inverse and the dual index. Liu and Maa \cite{liu} discussed the dual-core generalized inverse for third-order tensors in the framework of $T$-product. The dual group generalized inverse's existence, computation, and applications were examined by Zhong et al. \cite{dualgroup}.

Recently, Zhong and Zhang \cite{dualDrazin} established the concept of dual Drazin inverse (in short DDGI) for dual matrices of arbitrary index, not necessarily one. They also discussed the existence and computational aspects of the dual Drazin inverse, least square solutions, and minimal norm properties of DDGI.  In \cite{dualC-Ndecomp}, the authors introduced the Dual Core nilpotent decomposition (shortly C-N decomposition) and binary relation. Wei et al. \cite{wei} studied the perturbation of Drazin inverse and dual Drazin inverse. The existence of dual $r$-rank decomposition and a few properties have been studied in \cite{wang-rr-deco}.  

Motivated by the work of the dual Drazin inverse \cite{dualDrazin}, we extend the notion of dual-core generalized inverse to dual core-EP (in short CEP) generalized inverse of arbitrary index. Since the core-EP decomposition of a real matrix of arbitrary index always exists, we aim to introduce a new dual matrix decomposition, i.e., the dual core-EP decomposition of a real dual matrix.  
The main contribution of the manuscript as follows:
\begin{itemize}
    \item We introduce the notion of dual CEP generalized inverse for dual matrices. The existence and uniqueness of DCEPGI for square dual matrices are also discussed. 
    \item A few characterizations on dual CEP generalized inverse are discussed.
    \item We introduce dual core-EP decomposition for square real dual matrices of arbitrary index.
    \item A compact formula for computing aspects and relationships with other dual generalized inverses is established.
    \item An application in solving systems of inconsistent linear dual equations is discussed.  
\end{itemize}
The paper is organized as follows. In Section 2, we give some existing definitions and known results of real as well as dual matrices. The existence and uniqueness of dual CEP inverse and its characterizations are presented in Section 3. Section 4 contains the introduction of dual core-EP decomposition for dual matrix along with a compact formula derived. In Section 5, we present the relationships with other dual generalized inverses and establish an application for solving a dual linear system. In Section 6, we conclude with future research problems.

\section{Results}\label{sec2}

First, we present some useful notations and definitions of several generalized inverses. The sets of all dual real and real $m\times n$ matrices are represented by $\mathbb{D}\mathbb{R}^{m\times n}$ and $\mathbb{R}^{m\times n}$, respectively. For $A\in \mathbb{R}^{m\times n}$, the range space, and null space of dual matrix $\widehat{A}$, respectively are denoted by $\rg(\widehat{A}),$ and $\nl(\widehat{A})$.  We denote $I$ and $O$ as the identity and zero matrices of suitable order, respectively, and projection on subspace $F$ along subspace $G$ by $P_{F, G}$ whenever $F$ and $G$ are the direct sum of the entire vector space. The smallest non-negative integer $s$ such that rank$(A^{s+1})$=rank$(A^{s})$, is called the index of  $A$.

A dual matrix $\widehat{A}$ is defined as $\widehat{A}=A+\epsilon A_{0}$, where real matrices $A$ and $A_{0}$ are called the standard and infinitesimal parts of $\widehat{A}$. The dual matrix $\widehat{A}$ is called appreciable if $A\neq O$, otherwise, $\widehat{A}$ is called infinitesimal. If $\widehat{A}=A+ \epsilon A_{0} \in \mathbb{D}\mathbb{R}^{n \times n}$, then the dual transpose is defined as $\widehat{A}^{T}=A^{T}+\epsilon A_{0}^{T}$. 
%For a real unitary matrix $V,~V\widehat{A}V^{T}=VAV^{T}+\epsilon VA_{0}V^{T}$.
Next, we use matrix equations to define the dual generalized inverse for the dual matrices. 
\[\text{(1)}~\,\widehat{A}\widehat{X}\widehat{A}=\widehat{A},~\,\text{(2)}~\,\widehat{X}\widehat{A}\widehat{X}=\widehat{X},\,\text{(3)}~\,(\widehat{A}\widehat{X})^{T}=\widehat{A}\widehat{X},~\,\text{(4)}~\,(\widehat{X}\widehat{A})^{T}=\widehat{X}\widehat{A},\]
\[\text{(5)}~\,\widehat{A}\widehat{X}=\widehat{X}\widehat{A},~\,\mbox{($1^k$)}~\,\widehat{X}\widehat{A}^{k+1}=\widehat{A}^k,~\,\text{(6)}~\,\widehat{A}\widehat{X}^2=\widehat{X}.\]
%, as shown in Table \ref{tab1}.
%\vskip -0.2cm
%  \begin{table}[h!]
%    \centering
%        \caption{Equations in defining dual generalized inverses }
 %       \vspace{0.2cm}
%       \renewcommand{\arraystretch}{1.6} 
%    \begin{tabular}{|c|c|c|c|c|}
 %   \hline
 %   Equation  & ($1$) & ($2$) & ($3$) & ($4$) \\ 
 %        \hline
 %     Equation  &$\widehat{A}\widehat{X}\widehat{A}=\widehat{A}$ & $\widehat{X}\widehat{A}\widehat{X}=\widehat{X}$ & $(\widehat{A}\widehat{X})^{T}=\widehat{A}\widehat{X}$ & $(\widehat{X}\widehat{A})^{T}=\widehat{X}\widehat{A}$ \\
%            \hline
%      Label  & ($5$) & ($1^k$) & ($6$) &\\
%         \hline
  %    Equation  & $\widehat{A}\widehat{X}=\widehat{X}\widehat{A}$ &$\widehat{X}\widehat{A}^{k+1}=\widehat{A}^k$ & $\widehat{A}\widehat{X}^2=\widehat{X}$ &\\
   %      \hline
  %  \end{tabular}
 %   \label{tab1}
%\end{table}  
Let $\Omega$ be a nonempty subset of $\{1,2,3,4,5,1^k,6\}$. If a dual matrix $\widehat{X}\in \mathbb{D}\mathbb{R}^{n\times m}$ satisfies equation ($j$) for each $j\in \Omega$ then $\widehat{X}$ is called a $\Omega$-inverse of $\widehat{A}$. We denote such an inverse by $\widehat{A}^{(j)}$  and the set of all $\Omega$-inverse of $\widehat{A}$ by $\widehat{A}\{\Omega\}$. Using these representations, we  define the dual generalized inverses as follows: 
\begin{definition}\cite{pennestri2018moore}
Let $\widehat{A}\in \mathbb{D}\mathbb{R}^{m\times n}$. A matrix $\widehat{X} \in \mathbb{D}\mathbb{R}^{n\times m}$ is called 
\begin{enumerate}
    \item[(a)] an inner dual generalized inverse of $\widehat{A}$ if $\widehat{X}\in \widehat{A}\{1\}$ and is denoted by $\widehat{A}^{(1)}$. 
    \item[(b)] an outer dual generalized inverse of $\widehat{A}$ if $\widehat{X}\in \widehat{A}\{2\}$ and is denoted by $\widehat{A}^{(2)}$. 
    \item[(c)] the dual Moore-Penrose generalized inverse of $\widehat{A}$ if $\widehat{X}\in \widehat{A}\{1,2,3,4\}$ and is denoted by $\widehat{A}^{\dagger}$.
\end{enumerate}
\end{definition}
The Moore-Penrose dual generalized inverse (shortly MPDGI) of a dual matrix $\widehat{A}=A+\epsilon A_{0}$ was introduced by Pennestri et al. \cite{lineraalgorithm}, is denoted by $\widehat{A}^{P}$ and can be computed as $\widehat{A}^{P}=A^{\dagger}-\epsilon A^{\dagger}A_{0}A^{\dagger}$.
\begin{definition}\cite{dualDrazin}
Let $\widehat{A}\in \mathbb{D}\mathbb{R}^{n\times n}$ and $\text{Ind}(A)=m$. If $\widehat{X}\in \widehat{A}\{1^m,2,5\}$, then $\widehat{X}$ is called the dual Drazin generalized inverse of $\widehat{A}$ and is denoted by $\widehat{A}^{D}$.
\end{definition}
The inverse $\widehat{A}^{D}$ is unique if it exists. When $\text{Ind}(A)=1$, it is called the dual group generalized inverse, denoted by $\widehat{A}^{\#}$ \cite{dualgroup}. The following lemma is a result similar to that of \cite{wang2021characterizations}.
\begin{lemma}\label{lem2.3}
    Consider $\widehat{A}=A+\epsilon B \in \mathbb{D}\mathbb{R}^{n \times n}$ with Ind$(A)=m$. Denote $\widehat{A}^{m}={A}^{m}+\epsilon S$, where $S:=\sum_{i=1}^{m} A^{m-i}BA^{i-1}$. Then DMPGI of $\widehat{A}$ exists, and 
    \begin{center}
        $(\widehat{A}^{m})^{\dagger}=(\widehat{A}^{m})^{P}\iff(I-\widehat{A}^{m}(\widehat{A}^{m})^{\dagger})S=O$ and $(I-(\widehat{A}^{m})^{\dagger}\widehat{A}^{m})S=O$.
    \end{center}
    \end{lemma}
\begin{lemma}\cite{wang2021characterizations}
    Consider $\widehat{A}=A+\epsilon A_0 \in \mathbb{D}\mathbb{R}^{n \times n}$. Then the following  assertions are equivalent:
    \begin{enumerate}[\rm(i)]
    \item $(I-AA^{\dagger})A_0(I-A^{\dagger}A)=O$,
    \item rank$\left(\begin{bmatrix}
    A_0 & A \\A & O
    \end{bmatrix}\right)$= $2$ rank$(A)$,
    ${\widehat{A}}^{\dagger}$ exists.
\end{enumerate}
Moreover, $\widehat{A}^{\dagger}=A^{\dagger}+\epsilon \left(-A^{\dagger}A_0A^{\dagger}+(A^{T}A)^{\dagger}A_0^{T}(I-AA^{\dagger})+(I-A^{\dagger}A)A_0^{T}(AA^{T})^{\dagger}\right)$ if $\widehat{A}^{\dagger}$ exists.
\end{lemma}
\begin{lemma}\cite{dualDrazin}
    Consider $\widehat{A}=A+\epsilon B \in \mathbb{D}\mathbb{R}^{n \times n}$ with  Ind$(A)=m$. Denote $\widehat{A}^{m}={A}^{m}+\epsilon S$, where $S:=\sum_{i=1}^{m} A^{m-i}BA^{i-1}$. Then the following  assertions are equivalent:
    \begin{enumerate}[\rm(i)]
        \item $\widehat{A}^{D}$ exists,
        \item $\widehat{A}=P\begin{bmatrix}
T & O \\O & N
\end{bmatrix}P^{-1}+\epsilon P\begin{bmatrix}
B_{1} & B_{2}\\B_{3} & B_{4}
\end{bmatrix}P^{-1}$ where  $N^{m}=O$, $\sum_{i=1}^{m} N^{m-i}B_4 N^{i-1}=O$, and $T^{-1}$ exists. 
\item $(I-AA^{D})S(I-AA^{D})=O$,
\item rank$\left(\begin{bmatrix}
S & A^{m} \\A^{m} & O
\end{bmatrix}\right)$=$2$ rank$(A^{m})$,
\item $(\widehat{A}^{m})^{\dagger}$ exists.
\end{enumerate}
Moreover, if $\widehat{A}^{D}$ exists, then \begin{equation}\label{A^D}
\widehat{A}^{D}=A^D+\epsilon \left(-A^{D}BA^{D}+\sum_{i=0}^{m-1} (A^{D})^{i+2}BA^{i}(I-AA^{D})+\sum_{i=0}^{m-1}(I-AA^{D})A^{i}B(A^{D})^{i+2}\right).
\end{equation}

\end{lemma}

\begin{lemma}\label{lemcepdecomp}\cite{WangcEPD}
Consider $A \in \mathbb{R}^{n \times n}$ with rank$(A^m)=r$ and Ind$(A)=m$. Then $A=A_1+N$, where $A_1$ and $N$ satisfies 
\begin{enumerate}[\rm(i)]
    \item $A_1^{T}N=NA_1=O$;
    \item Ind$(A_1)=1$;
    \item $N^m=O$.
    \end{enumerate}
Mooreover, the decomposition is unique and \begin{equation} \label{eqncEPD}
T= U \begin{bmatrix}
T_{1} & T_{2}\\
O & O
\end{bmatrix} U^{T},~
N=U\left[\begin{array}{cc}
O & O \\
O & N
\end{array}\right] U^{T},~ A= U \begin{bmatrix}
T_{1} & T_{2} \\
O & N
\end{bmatrix} U^{T},
 \end{equation}
where $U \in \mathbb{R}^{n \times n}$ is unitary, $T_{1} \in \mathbb{R}^{r \times r}$ is nonsingular,  and $N \in \mathbb{R}^{(n-r) \times(n-r)}$ is nilpotent with nilpotency index $m$. 
%Additionally, $T$ and $N$ are called core part and nilpotent part of $A$, respectively.
\end{lemma} 
Consider $A$ as in the form of \eqref{eqncEPD}, then from \cite[Theorem 3.7]{ferryaDrazin}, \begin{equation}\label{Drazincepdcmp}
    A^{D}= U \begin{bmatrix}
T_{1}^{-1} & (T_{1}^{m+1})^{-1}\Tilde{T}\\
O & O
\end{bmatrix} U^{T},\end{equation} where $\Tilde{T}=\sum_{i=0}^{m-1} T_{1}^{i}T_{2}N^{m-i-1}$.
Also, from \cite{WangcEPD}, \begin{equation}\label{cepdcmp}
    A^{\ep}= U \begin{bmatrix}
T_{1}^{-1} & O\\
O & O
\end{bmatrix} U^{T}.\end{equation} For $B_1 \in \mathbb{R}^{r \times r}$, $B_4 \in \mathbb{R}^{(n-r) \times(n-r)}$, the expression of 
\begin{equation}\label{Bdcomp}
    U^{T}BU=\begin{bmatrix}
B_{1} & B_{2}\\
B_{3} & B_{4}
\end{bmatrix}
\end{equation}
yields \begin{equation}\label{Ahatdcomp}
    \widehat{A}= U \begin{bmatrix}
T_{1} & T_{2}\\
O & N
\end{bmatrix} U^{T}+\epsilon U \begin{bmatrix}
B_{1} & B_{2}\\
B_{3} & B_{4}
\end{bmatrix} U^{T}.
\end{equation}
\begin{lemma}(\cite{Gao2017})\label{lem2.8}
Consider  $A \in \mathbb{R}^{n \times n}$ with $\text{Ind}(A)=m$. Then $A^{\ep}=A^{D}A^{m}(A^{m})^{\dagger}$.
\end{lemma}

Next, we recall three important definitions that have already  been specified in \cite{dualC-Ndecomp}.
\begin{definition}
    Let $A$ be of the form \eqref{eqncEPD}, $\widehat{A}=A+\epsilon B \in \mathbb{D}\mathbb{R}^{n \times n}$ and Ind$(A)=m$.  The appreciable index of $\widehat{A}$ is the index of the standard part $A$ and is denoted by $A\text{Ind}(\widehat{A})$.  
\end{definition}
\begin{definition}
    Let $\widehat{A}=A+\epsilon B \in \mathbb{D} \mathbb{R}^{n \times n}$ with $A\text{Ind}(\widehat{A})=m$ and rank$(A^{m})=t$. Denote $\widehat{A}^{m}={A}^{m}+\epsilon S$, where $S:=\sum_{i=1}^{m} A^{m-i}BA^{i-1}$. The dual index of $\widehat{A}$ is the least positive integer $k~(m\leq k \leq 2m)$ such that $$A\text{rank}(\widehat{A}^{k})=\text{ rank}(\widehat{A}^{k}).$$
\end{definition}
\begin{definition}
    Let $\widehat{N}=N+\epsilon N_{0} \in \mathbb{D} \mathbb{R}^{n \times n}$. Then $\widehat{N}$ is called $k$-dual nilpotent matrix if $\widehat{N}^{k}=O$ for a  positive integer $k$.
\end{definition}
Using the analogy of the real case \cite{PrasadMo14}, the dual core-EP inverse can be defined as:
\begin{definition}
    Let $\widehat{A}=A+\epsilon B \in \mathbb{D} \mathbb{R}^{n \times n}$ with $\text{Ind}(A)=m$. Then $\widehat{X} \in \mathbb{D R}^{n \times n}$ is called the dual CEP inverse of $\widehat{A}$ if it satisfies $$(\widehat{A}\widehat{X})^{T}=\widehat{A}\widehat{X},~\widehat{A}\widehat{X}^{2}=\widehat{X},~\widehat{X}\widehat{A}^{m+1}=\widehat{A}^{m},$$and is denoted by $\widehat{A}^{\ep}$.
\end{definition}

\section{Dual Core-EP Generalized inverses}\label{sec3}

In this section, we introduce the concept of DCEP generalized inverses specifically for dual matrices and explore various characterizations of these inverses. Additionally, we examine the connections between DCEP generalized inverses and other forms of dual generalized inverses.
\begin{lemma}\label{lemsyst1}
    Let $\widehat{A}=A+\epsilon B \in \mathbb{D} \mathbb{R}^{n \times n}$ and Ind$(A)=m$. Denote $\widehat{A}^{m}={A}^{m}+\epsilon S$, where $S:=\sum_{i=1}^{m} A^{m-i}BA^{i-1}$. Then  $\widehat{X}= X+\epsilon R \in \mathbb{D} \mathbb{R}^{n \times n}$ is called   dual CEP inverse of $\widehat{A}$ if and only if $X=A^{\ep}$ and 
    \begin{equation}\label{syst1}
        \begin{alignedat}{3}
(AR+BA^{\ep})^{T}=AR+BA^{\ep},\\B(A^{\ep})^{2}+AA^{\ep}R+ARA^{\ep}=R,\\RA^{m+1}+A^{\ep}AS+A^{\ep}BA^{m}=S.
    \end{alignedat}
    \end{equation}
\end{lemma}

\begin{proof}
It follows from the definition of the dual CEP inverse that $\widehat{X}= X+\epsilon R$ is called the DCEPGI of $\widehat{A}$ if and only if  
\begin{align*}
    ((A+\epsilon B)(X+\epsilon R))^{T}=(A+\epsilon B)(X+\epsilon R),\\(A+\epsilon B)(X+\epsilon R)^{2}=(X+\epsilon R),\\(X+\epsilon R)(A+\epsilon B)^{m+1}=(A+\epsilon B)^{m}.
\end{align*} 
By using $\widehat{A}^{m}={A}^{m}+\epsilon S$, where $S:=\sum_{i=1}^{m} A^{m-i}BA^{i-1}$, the above three equations can be reduced to 
\begin{align*}
    (AX)^{T}+\epsilon(AR+BX)^{T}=AX+\epsilon(AR+BX),\\AX^{2}+\epsilon(BX^{2}+AXR+ARX)=X+\epsilon R,\\XA^{m+1}+\epsilon(RA^{m+1}+XAS+XBA^{m})={A}^{m}+\epsilon S.
\end{align*} Therefore, it is clear from the prime parts of above equations that $(AX)^{T}=AX,~AX^{2}=X,~XA^{m+1}=A^{m}$ i.e., $X=A^{\ep}$. Furthermore, it can be noted from the dual parts of the above equations that system \ref{syst1} satisfied.
\end{proof}
In \cite{PrasadMo14}, the authors showed that the CEP inverse for real square matrices  always exists and is unique. However, it is possible that a dual square matrix may not always have a DCEPGI. The following example demonstrates that a square dual matrix does not have a DCEP generalized inverse.

\begin{example}\rm
Let  $\widehat{A}=\begin{bmatrix}
1 &2 & -2 \\0 &0 & -2\\0 &0 & 0
\end{bmatrix}+\epsilon\begin{bmatrix}
1 & 5 & -2 \\0 &3 & -2\\2 &0 & 4
\end{bmatrix}$. Then $A^{\ep}=\begin{bmatrix}
1 & 0&0 \\0 &0 & 0\\0 &0 & 0
\end{bmatrix}$. If $\widehat{A}^{\ep}$ exists, then by Lemma \ref{lemsyst1}, it can be written as  $\widehat{A}^{\ep}=\begin{bmatrix}
1 & 0&0 \\0 &0 & 0\\0 &0 & 0
\end{bmatrix}+\epsilon\begin{bmatrix}
r_{1} & r_{2}&r_{3} \\r_{4} &r_{5} & r_{6}\\r_{7} &r_{8} & r_{9}
\end{bmatrix}$. From direct calculation, it shows that $RA^{3}+A^{\ep}AS+A^{\ep}BA^{2}=\begin{bmatrix}
r_{1}-13 & 2r_{1}+7 & -6r_{1}-52\\ r_{4} & 2r_{4} & -6r_{4}\\ r_{7} &2r_{7}&-6r_{7} \end{bmatrix}\neq \begin{bmatrix}
-2 & 13& -26 \\-4 &0 & -14\\2 &4 & -4
\end{bmatrix}=S$. Thus, it does not hold third condition of system \eqref{syst1} since $r_{4}=-4$ and $r_{4}=0$, which is not possible hence $\widehat{A}^{\ep}$ does not exist for this matrix.
\end{example}
Next, we show the uniqueness of the dual core-EP inverse of a square dual matrix.
\begin{theorem}
  Consider $\widehat{A}=A+\epsilon B \in \mathbb{D}\mathbb{R}^{n \times n}$ and Ind$(A)=m$. If $\widehat{A}^{\ep}$ exists then it is  unique.
\end{theorem}
\begin{proof}
    It follows from Lemma \ref{lemsyst1} that if the dual CEP inverse of $\widehat{A}$ exists then it is of the form $A^{\ep}+\epsilon R$ with $R$ satisfying \eqref{syst1}. To show uniquness, we assume $\widehat{X_{1}}=A^{\ep}+\epsilon R_{1}$ and $\widehat{X_{2}}=A^{\ep}+\epsilon R_{2}$, two dual CEP inverses of $\widehat{A}$. Denote
    $$ Z=R_{1}-R_{2}. $$
Next, we show $Z=O$.
From the first equation of system \eqref{syst1}, it can be observed that 
\begin{equation}\label{1steqn}
    \left\{\begin{array}{l}
(AR_{1}+BA^{\ep})^{T}=AR_{1}+BA^{\ep}, \\
(AR_{2}+BA^{\ep})^{T}=AR_{2}+BA^{\ep}.
\end{array}\right.
\end{equation}

By subtracting second equation from the first equation in \eqref{1steqn}, we obtain
\begin{equation}\label{eqn4.3}
    (AZ)^{T}=AZ.
\end{equation}
From second equation of system \eqref{syst1}, we obtain
\begin{equation}\label{2ndeqn}
\left\{\begin{array}{l}
R_{1}=B(A^{\ep})^{2}+AA^{\ep}R_{1}+AR_{1}A^{\ep}, \\
R_{2}=B(A^{\ep})^{2}+AA^{\ep}R_{2}+AR_{2}A^{\ep}.
\end{array}\right.
\end{equation}  Subtracting the second equation from the first equation in \eqref{2ndeqn}, we get \begin{equation}\label{2.1eqn}
    Z=AA^{\ep}Z+AZA^{\ep}.
\end{equation}
Similarly, we obtain from third equation of system \eqref{syst1}
\begin{equation}\label{3rdeqn}
\left\{\begin{array}{l}
S=R_{1}A^{m+1}+A^{\ep}AS+A^{\ep}BA^{m}, \\
S=R_{2}A^{m+1}+A^{\ep}AS+A^{\ep}BA^{m}.
\end{array}\right.
\end{equation}  Subtracting the second equation from the first equation in \eqref{3rdeqn}, we obtain \begin{equation}\label{4.7eqn}
    O=ZA^{m+1}.
\end{equation} By pre-multiplying $A$ and post-multiplying by $(A^{\ep})^{m+1}$ both sides of \eqref{4.7eqn}, we get $$O=AZA^{m+1}(A^{\ep})^{m+1}=AZAA^{\ep},$$ and then multiplying $A^{\ep}$ on the right, we obtain $$AZA^{\ep}=O.$$
Therefore,  equation \eqref{2.1eqn} becomes $Z=AA^{\ep}Z$.
Let the decomposition of $A$ be as in equation \eqref{eqncEPD}. Let
\begin{equation}\label{Zexpes}
Z=U\left[\begin{array}{ll}
Z_1 & Z_2 \\
Z_3 & Z_4
\end{array}\right] U^{T},
\end{equation}
where $Z_1 \in \mathbb{R}^{r \times r}$. Using \eqref{eqncEPD},\eqref{cepdcmp} and \eqref{Zexpes} in \eqref{2.1eqn}, we obtain
$$ 
\begin{aligned}
AA^{\ep}Z & =U\left[\begin{array}{cc}
T_{1} & T_{2} \\
O & N
\end{array}\right] U^{T} U\left[\begin{array}{cc}
T_{1}^{-1} & O \\
O & O
\end{array}\right] U^{T} U\left[\begin{array}{ll}
Z_1 & Z_2 \\
Z_3 & Z_4
\end{array}\right] U^{T} \\
& =U\left[\begin{array}{cc}
I_r & O \\
O & O
\end{array}\right] U^{T} U\left[\begin{array}{ll}
Z_1 & Z_2 \\
Z_3 & Z_4
\end{array}\right] U^{T} \\
& =U\left[\begin{array}{cc}
Z_1 & Z_2 \\
O & O
\end{array}\right] U^{T}=U\left[\begin{array}{ll}
Z_1 & Z_2 \\
Z_3 & Z_4
\end{array}\right] U^{\mathrm{T}}=Z.
\end{aligned}
$$

Therefore, $Z_3=Z_4=O$, i.e.,
$$
Z=U\left[\begin{array}{cc}
Z_1 & Z_2 \\
O & O
\end{array}\right] U^{T} .
$$
From equation \eqref{eqn4.3}, we obtain
$$
\begin{aligned}
(AZ)^{T} & =\left(U\left[\begin{array}{cc}
T_1 & T_2 \\
O & N
\end{array}\right] U^{T} U\left[\begin{array}{cc}
Z_1 & Z_2 \\
O & O
\end{array}\right] U^{T}\right)^{T} \\
& =\left(U\left[\begin{array}{cc}
T_{1}Z_1 & T_{1}Z_2 \\
O & O
\end{array}\right] U^{T}\right)^{T} \\
& =U\left[\begin{array}{cc}
(T_{1}Z_1)^{T} & O \\
(T_{1}Z_2)^{T} & O
\end{array}\right] U^{T}=U\left[\begin{array}{cc}
T_{1}Z_1 & T_{1}Z_2 \\
O & O
\end{array}\right] U^{T}.
\end{aligned}
$$
Therefore, $T_{1}Z_2=O$ and hence we obtain $Z_2=O$.
So,$$ Z=U\left[\begin{array}{cc}
Z_{1} & O \\
O & O
\end{array}\right] U^{\mathrm{T}}.$$ Again from equation \eqref{4.7eqn}, we obtain 
$$
\begin{aligned}
O & =U\left[\begin{array}{cc}
Z_1 & O \\
O & O
\end{array}\right] U^{T} U\left[\begin{array}{cc}
T_1^{m+1} & T_{1}\Tilde{T} \\
O & O
\end{array}\right] U^{T} \\
& =U\left[\begin{array}{cc}
Z_{1}T_1^{m+1} & Z_{1}T_{1}\Tilde{T} \\
O & O
\end{array}\right] U^{T},
\end{aligned}
$$ where $\Tilde{T}=\sum_{i=0}^{m-1} T_{1}^{i}T_{2}N^{m-1-i}$. So, $Z_{1}=0$. Hence, we obtain $Z=O$ i.e., $R_{1}=R_{2}$. Therefore, we proved the existence and uniqueness of $\widehat{A}^{\ep}$.
\end{proof}
In the following theorem, we provide necessary and sufficient criteria for the existence of the dual CEPGI using core-nilpotent decomposition.

\begin{theorem}\label{equivalent}
    Let $\widehat{A}=A+\epsilon B \in \mathbb{D}\mathbb{R}^{n \times n}$ with  Ind$(A)=m$. Denote $\widehat{A}^{m}={A}^{m}+\epsilon S$, where $S:=\sum_{i=1}^{m} A^{m-i}BA^{i-1}$. Then the subsequent statements  are equivalent:
    \begin{enumerate}[\rm(i)]
        \item  $\widehat{A}^{\ep}$ exists;
        \item $\widehat{A}=U\begin{bmatrix}
T_{1} & T_2 \\O & N
\end{bmatrix}U^{T}+\epsilon U\begin{bmatrix}
B_{1} & B_{2}\\B_{3} & B_{4}
\end{bmatrix}U^{T}$ where $U$ is unitary matrix and $T_{1}$ is nonsingular matrix, $N^{m}=O$, and $\sum_{i=1}^{m} N^{m-i}B_3 T_1^{i-m-1}\Tilde{T}=\sum_{i=1}^{m} (N^{m-i}B_3F+N^{m-i}B_4 N^{i-1})$, where $\Tilde{T}= \sum_{i=0}^{m} T_{1}^{i}T_{2}N^{m-i}$ and $F= \sum_{j=0}^{i-2} T_{1}^{j}T_{2}N^{i-2-j}$.
\item $\left(I-A^{m}(A^{m})^{\ep}\right)S\left(I-(A^{m})^{\ep}A^{m}\right)=O$
\end{enumerate}

\end{theorem}
  
\begin{proof}
(i)$\Rightarrow$(ii):  Let $A$ and $A^{\ep}$ be of the form \eqref{eqncEPD} and \eqref{cepdcmp}, respectively and Ind$(A)=m$. Let $B=U\left[\begin{array}{ll}B_1 & B_2 \\ B_3 & B_4\end{array}\right] U^{T}$ and $R=U\left[\begin{array}{ll}R_1 & R_2 \\ R_3 & R_4\end{array}\right] U^{T}$. Then

\begin{equation}\label{Sdcomp}
\begin{split}
&
S= \sum_{i=1}^m A^{m-i}BA^{i-1} \\
&  =
U\begin{bmatrix}
\sum_{i=1}^{m}(T_{1}^{m-i}B_1+EB_{3})T_{1}^{i-1} & \sum_{i=1}^{m}\left((T_{1}^{m-i}B_1+EB_3)F+(T_1^{m-i}B_2+EB_4)\right)N^{i-1} \\
\sum_{i=1}^{m} N^{m-i}B_3T_{1}^{i-1} & \sum_{i=1}^{m} (N^{m-i}B_3F+N^{m-i}B_4 N^{i-1})
\end{bmatrix} U^{T},
\end{split}
\end{equation} where $E= \sum_{j=0}^{m-i} T_{1}^{j}T_{2}N^{m-i-j}$ and $F= \sum_{j=0}^{i-2} T_{1}^{j}T_{2}N^{i-2-j}$.

Take $S_1=\sum_{i=1}^{m}(T_{1}^{m-i}B_1+EB_{3})T_{1}^{i-1},~ S_2 \sum_{i=1}^{m}[(T_{1}^{m-i}B_1+EB_3)F+(T_1^{m-i}B_2+EB_4)]N^{i-1},\\ S_3= \sum_{i=1}^{m} N^{m-i}B_3T_{1}^{i-1},~S_4=\sum_{i=1}^{m} N^{m-i}B_3F+N^{m-i}B_4 N^{i-1}$.
Since the dual CEP generalized inverse of $\widehat{A}$ exists, it follows from the third condition of equation \eqref{syst1} that
\begin{equation*}
\begin{split}
& \begin{bmatrix}
S_1 & S_2 \\
S_3 & S_4
\end{bmatrix} \\
&\qquad =\begin{bmatrix}
R_{1}T_{1}^{m+1}& R_{1}T_{1}\Tilde{T} \\
R_{3}T_{1}^{m+1} & R_{3}T_{1}\Tilde{T}
\end{bmatrix}+ \begin{bmatrix}
S_1+T_1^{-1}T_2 S_3 & S_2+T_1^{-1}T_2 S_4 \\
O & O
\end{bmatrix}+ \begin{bmatrix}
T_{1}^{-1}B_1T_{1}^{m} & T_{1}^{-1}B_1\Tilde{T}\\
O & O
\end{bmatrix} \\
&\qquad =\begin{bmatrix}
R_{1}T_{1}^{m+1}+S_1+T_1^{-1}T_2 S_3+T_{1}^{-1}B_1T_{1}^{m} & R_{1}T_{1}\Tilde{T}+S_2+T_1^{-1}T_2 S_4+T_{1}^{-1}B_1\Tilde{T} \\
R_{3}T_{1}^{m+1} & R_{3}T_{1}\Tilde{T}
\end{bmatrix}.
\end{split}
\end{equation*}
One can obtain from the above equality that $\sum_{i=1}^{m} N^{m-i}B_3 T_1^{i-m-1}\Tilde{T}=\sum_{i=1}^{m} (N^{m-i}B_3F+N^{m-i}B_4 N^{i-1})$.

(ii)$\Rightarrow$(iii): If (ii) holds, then \begin{align*}\begin{split}
    & \left(I-A^{m}(A^{m})^{\ep}\right)S\left(I-(A^{m})^{\ep}A^{m}\right)\\ & \qquad=U\begin{bmatrix}
O & O \\
O & \sum_{i=1}^{m} N^{m-i}B_3F+N^{m-i}B_4 N^{i-1}-\sum_{i=1}^{m} N^{m-i}B_3T_{1}^{i-m-1}\Tilde{T}
\end{bmatrix}U^{T}\\ & \qquad =O.
\end{split}\end{align*}

(iii)$\Rightarrow$(i): Suppose (iii) holds. Write 
\begin{align*}
\begin{split}  & \widehat{X}=U\begin{bmatrix}T_{1}^{-1} & O \\ O & O\end{bmatrix} U^{T} + \\ & \epsilon U\begin{bmatrix}
-T_{1}^{-1}B_{1}T_{1}^{-1}-T_{1}^{-1}T_{2}\sum_{i=1}^{m}N^{m-i}B_3T_{1}^{i-m-2} & T_{1}^{-1}(\sum_{i=1}^{m}N^{m-i}B_3T_{1}^{i-m-1})^{T}\\\sum_{i=1}^{m}N^{m-i}B_3T_{1}^{i-m-2} & O
\end{bmatrix}U^{T}.
   \end{split} 
   \end{align*}
Now we verify the conditions in the definition of dual CEP generalized inverse. 

\begin{equation*}
\begin{split}
& \widehat{X}\widehat{A}^{m+1} \\ & \qquad  =U\begin{bmatrix}
T_{1}^{-1} & O \\O & O
\end{bmatrix}U^{T}+ \\ & \qquad \epsilon U\begin{bmatrix}
-T_{1}^{-1}B_{1}T_{1}^{-1}-T_{1}^{-1}T_{2}\sum_{i=1}^{m}N^{m-i}B_3T_{1}^{i-m-2} & T_{1}^{-1}(\sum_{i=1}^{m}N^{m-i}B_3T_{1}^{i-m-1})^{T}\\ \sum_{i=1}^{m}N^{m-i}B_3T_{1}^{i-m-2} & O
\end{bmatrix}U^{T} \times\\
& \qquad U\begin{bmatrix}
T_{1}^{m+1} & T_1\Tilde{T} \\
O & O
\end{bmatrix}U^{T}+ \epsilon U \\ & \qquad \begin{bmatrix}
\sum_{i=1}^{m}(T_{1}^{m-i}B_1+EB_3)T_1^{i-1} & \sum_{i=1}^{m}\left((T_{1}^{m-i}B_1+EB_3)F+(T_{1}^{m-i}B_2+EB_4)N^{i-1}\right) \\
\sum_{i=1}^{m} N^{m-i}B_3T_1^{i-1} & \sum_{i=1}^{m} N^{m-i}(B_3F+B_4N^{i-1})
\end{bmatrix} U^{T} \\
& \qquad =U\begin{bmatrix}
T_{1}^{m} & \Tilde{T} \\
O & O
\end{bmatrix} U^{T}+ \epsilon U \\
& \qquad \begin{bmatrix}
\sum_{i=1}^{m}(T_{1}^{m-i}B_1+EB_{3})T_{1}^{i-1} & \sum_{i=1}^{m}\left((T_{1}^{m-i}B_1+EB_3)F+(T_1^{m-i}B_2+EB_4)\right)N^{i-1} \\
\sum_{i=1}^{m} N^{m-i}B_3T_{1}^{i-1} & \sum_{i=1}^{m} (N^{m-i}B_3F+N^{m-i}B_4 N^{i-1})
\end{bmatrix} U^{T} \\
& \qquad =\widehat{A}^{m}.
\end{split}
\end{equation*}
Since 

\begin{equation*}
\begin{split}
& \widehat{A}\widehat{X}  =\left( U\begin{bmatrix}T_1 & T_2 \\ O & N\end{bmatrix} U^{T}+\epsilon U\begin{bmatrix}B_1 & B_2 \\ B_3 & B_4\end{bmatrix} U^{T}\right) \times \\ 
& \qquad  U\begin{bmatrix}T_{1}^{-1} & O \\ O & O\end{bmatrix} U^{T} + \\ & \qquad \epsilon U\begin{bmatrix}
-T_{1}^{-1}B_{1}T_{1}^{-1}-T_{1}^{-1}T_{2}\sum_{i=1}^{m}N^{m-i}B_3T_{1}^{i-m-2} & T_{1}^{-1}(\sum_{i=1}^{m}N^{m-i}B_3T_{1}^{i-m-1})^{T}\\\sum_{i=1}^{m}N^{m-i}B_3T_{1}^{i-m-2} & O
\end{bmatrix}U^{T} \\
& \qquad =U\begin{bmatrix}
I & O \\
O & O
\end{bmatrix} U^{T}+\epsilon U\begin{bmatrix}
O & (\sum_{i=1}^{m}N^{m-i}B_3T_{1}^{i-m-1})^{T} \\
\sum_{i=1}^{m} N^{m-i}B_3 T_{1}^{m-i-1} & O
\end{bmatrix} U^{T}\\ & \qquad =(\widehat{A}\widehat{X})^{T}.
\end{split}
\end{equation*}
Again 
\begin{align*}
\begin{split}
\widehat{A}\widehat{X}^2 & =\left(U\begin{bmatrix}
I & O \\
O & O
\end{bmatrix} U^{T}+\epsilon U\begin{bmatrix}
O & (\sum_{i=1}^{m}N^{m-i}B_3T_{1}^{i-m-1})^{T} \\
\sum_{i=1}^{m} N^{m-i}B_3 T_{1}^{m-i-1} & O
\end{bmatrix} U^{T}\right) \times \\ & \qquad U\begin{bmatrix}T_{1}^{-1} & O \\ O & O\end{bmatrix} U^{T} + \\ & \qquad \epsilon U\begin{bmatrix}
-T_{1}^{-1}B_{1}T_{1}^{-1}-T_{1}^{-1}T_{2}\sum_{i=1}^{m}N^{m-i}B_3T_{1}^{i-m-2} & T_{1}^{-1}(\sum_{i=1}^{m}N^{m-i}B_3T_{1}^{i-m-1})^{T}\\\sum_{i=1}^{m}N^{m-i}B_3T_{1}^{i-m-2} & O
\end{bmatrix}U^{T}\\ & \qquad 
U\begin{bmatrix}T_{1}^{-1} & O \\ O & O\end{bmatrix} U^{T} + \\ & \qquad \epsilon U\begin{bmatrix}
-T_{1}^{-1}B_{1}T_{1}^{-1}-T_{1}^{-1}T_{2}\sum_{i=1}^{m}N^{m-i}B_3T_{1}^{i-m-2} & T_{1}^{-1}(\sum_{i=1}^{m}N^{m-i}B_3T_{1}^{i-m-1})^{T}\\\sum_{i=1}^{m}N^{m-i}B_3T_{1}^{i-m-2} & O
\end{bmatrix}U^{T} \\ & \qquad =\widehat{X}.
\end{split}
\end{align*}

Therefore, $\widehat{A}^{\ep}$ exists and $\widehat{X}=\widehat{A}^{\ep}$.

\end{proof}

\section{Dual core-EP decomposition}\label{sec4}

This section introduces and establishes the uniqueness of the dual core-EP decomposition of a real dual matrix.
\begin{theorem}\label{decompprev}
   Consider $\widehat{A}=A+\epsilon B \in \mathbb{D}\mathbb{R}^{n \times n}$ with $A\text{Ind}(\widehat{A})=m$ and rank$(A^{m})=t$. There exists a unitary dual matrix $\widehat{U}=U+\epsilon U_{0} \in \mathbb{D}\mathbb{R}^{n \times n}$ such that 
    \begin{equation}\label{Ahatexprsn}
        \widehat{A}=\widehat{U}\begin{bmatrix}
\widehat{T_{1}} & \widehat{T_{2}} \\O & \widehat{N}
\end{bmatrix}\widehat{U}^{T}
    \end{equation}
\begin{equation}\label{4condneqn}\left\{\begin{array}{l}
\widehat{T_1}=T_{1}+\epsilon B_{1} \\
\widehat{T_2}=T_{2}+\epsilon B_{2} \\
\widehat{N}=N+\epsilon B_{4} \\ 
U_0=U\left[\begin{array}{cc}
O & U_2 \\
U_3 & O
\end{array}\right]
\end{array}\right.\end{equation}
$U \in \mathbb{R}^{n \times n},~T_1 \in \mathbb{R}^{t \times t}$ and $N \in \mathbb{R}^{(n-t) \times(n-t)}$ are as in Lemma \ref{lemcepdecomp}, $B_1, B_2, B_3$ and $B_4$ are as given in \eqref{Bdcomp}, and
\begin{equation}\label{condition}
\left\{\begin{array}{l}
T_1U_2=U_2N,\\
U_3=B_3T_1^{-1}+NB_3T_1^{-2}+N^2B_3T_1^{-3}+\cdots+T_1^{-m}B_2N^{m-1}, \\
T_2U_{3}=O,~U_{3}T_2=O.
\end{array}\right.
\end{equation}
\end{theorem}
\begin{proof} Let $\widehat{A}=A+\epsilon B \in \mathbb{D}\mathbb{R}^{n \times n}$ with rank$\left(A^m\right)=t$, $A\text{Ind}(\widehat{A})=m$,  $A$ as expressed in \eqref{eqncEPD}, and $U^{T}BU$ be partitioned as \eqref{Bdcomp}. Then
$$
\widehat{A}=U\left[\begin{array}{cc}
T_1 & T_2 \\
O & N
\end{array}\right] U^{T}+\epsilon U\left[\begin{array}{ll}
B_1 & B_2 \\
B_3 & B_4
\end{array}\right] U^{T} .
$$
Moreover, suppose $\widehat{U}=U+\epsilon U_0$, and $\widehat{T_1},~\widehat{T_2},~\widehat{N}$ and $U_0$ be as given in \eqref{4condneqn}. It is easy to verify that

\begin{align}\label{eqn3prodct}
\notag \widehat{U}^{T} & =U^{T}-\epsilon U^{T}U_0U^{T}, \\
\notag \widehat{U}^{T} \widehat{A} \widehat{U} & =\left({U}^{T}-\epsilon {U}^{T}U_0{U}^{T}\right)\left(A+\epsilon B\right)\left(U+\epsilon U_0\right) \\
& ={U}^{T}AU+\epsilon\left({U}^{T}AU_0+{U}^{T}BU-{U}^{T}U_0 {U}^{T}AU\right),
\end{align}

and

\begin{align}
\notag & {U}^{T}AU_0+{U}^{T}BU-{U}^{T}U_0{U}^{T}AU \\
\notag = & {\left[\begin{array}{ll}
T_1 & T_2 \\
O & N
\end{array}\right]\left[\begin{array}{cc}
O &  U_2 \\
U_3 & O
\end{array}\right]+\left[\begin{array}{ll}
B_1 & B_2 \\
B_3 & B_4
\end{array}\right]-\left[\begin{array}{cc}
O &  U_2 \\
U_3 & O
\end{array}\right]\left[\begin{array}{cc}
T_1 & T_2 \\
O & N
\end{array}\right] } \\
= & \left[\begin{array}{cc}
T_{2}U_3+B_1 & T_1U_2+B_{2}-U_2N \\
NU_3+B_3-U_3T_1 & B_{4}-U_3T_2
\end{array}\right].
\end{align}
Now the equation \eqref{eqn3prodct} becomes 
\begin{equation}\label{finaleqn}
\widehat{U}^{T} \widehat{A} \widehat{U}=\left[\begin{array}{cc}
T_1 & T_2 \\
O & N
\end{array}\right]+\epsilon\left[\begin{array}{cc}
T_{2}U_3+B_1 & T_1U_2+B_{2}-U_2N \\
NU_3+B_3-U_3T_1 & B_{4}-U_3T_2
\end{array}\right].
\end{equation}

Suppose that $T_{2}U_3=0,~U_3T_2=O,~T_1U_2=U_2N$ and $U_3$ are of the form \eqref{condition}. Since $A\text{Ind}(\widehat{A})=m$, then
$$
NU_3+B_3-U_3T_1=O.
$$
Therefore, we can write \eqref{finaleqn} as
$$\begin{aligned}
\widehat{U}^{T} \widehat{A} \widehat{U} &=  \begin{bmatrix}
T_1 & T_2 \\
O & N
\end{bmatrix}+\epsilon\begin{bmatrix}
B_1 & B_{2} \\
O & B_{4}
\end{bmatrix}\\ & =\begin{bmatrix}
    \widehat{T_{1}} & \widehat{T_{2}} \\O & \widehat{N}
\end{bmatrix},
\end{aligned}$$
which proves the result.
\end{proof}
\begin{example}\rm
   Let 
$\widehat{A}=A+\epsilon B:=\begin{bmatrix}
1 & 0 & 0 & 0 \\
0 & 0 & -1 & 3 \\
0 & 0 & 0 & -2 \\
0 & 0 & 0 & 0
\end{bmatrix}+\epsilon\begin{bmatrix}
1 & 0 & -1 & 1 \\
0 & 1 & -1 & 0 \\
0 & 3 & 0 & -2 \\
0 & 2 & 0 & -1
\end{bmatrix}$,
where $\text{Ind}(A)=3$.  There exists a unitary matrix $U=\begin{bmatrix}
1 & 0 & 0 & 0 \\
0 & -1 & 0 & 0 \\
0 & 0 & 0 & -1 \\ 
0 & 0 & 1 & 0
\end{bmatrix}$ such that
$$
U^{T}AU=\begin{bmatrix}
1 & 0 & 0 & 0 \\
0 & 0 & -3 & -1 \\
0 & 0 & 0 & 0 \\
0 & 0 & 2 & 0
\end{bmatrix}, \quad U^{T}BU=\begin{bmatrix}
1 & 0 & 1 & 1 \\
0 & 1 & 0 & -1 \\
0 & -2 & -1 & 0 \\
0 & 3 & 2 & 0
\end{bmatrix}.
$$

According to Theorem \ref{decompprev}, we obtain
$$
\widehat{T_{1}}=1+\epsilon, \quad \widehat{T_{2}}=\begin{bmatrix}
0 & \epsilon & \epsilon
\end{bmatrix}, \widehat{N}=\begin{bmatrix}
0 & -3 & -1 \\
0 & 0 & 0 \\
0 & 2 & 0
\end{bmatrix}+\epsilon \begin{bmatrix}
1 & 0 & -1 \\
-2 & -1 & 0 \\
3 & 2 & 0
\end{bmatrix}, \quad U_3=\begin{bmatrix}
0 \\
0 \\
0
\end{bmatrix}.
$$

Thus the decomposition of dual matrix $\widehat{A}$ is
$$
\widehat{A}=\widehat{U}\begin{bmatrix}
1+\epsilon & 0 & \epsilon & \epsilon \\
0 & \epsilon & -3 & -1-\epsilon \\
0 & -2\epsilon &  -\epsilon & 0 \\
0 & 3\epsilon & 2+2\epsilon & 0
\end{bmatrix} \widehat{U}^{T},
$$
where the dual matrix $\widehat{U}=\begin{bmatrix}1 & 0 & 0 & 0 \\ 0 & -1 & 0 & 0 \\ 0 & 0 & 0 & -1 \\ 0 & 0 & 1 & 0\end{bmatrix}+\epsilon \begin{bmatrix}0 & 0 & 0 & 0 \\ 0 & 0 & 0 & 0 \\ 0 & 0 & 0 & 0 \\ 0 & 0 & 0 & 0\end{bmatrix}$ is unitary.
\end{example}

\begin{theorem}\label{DCEpdecomp}(Dual core-EP decomposition)
    Let $\widehat{A}=A+\epsilon B \in \mathbb{D}\mathbb{R}^{n \times n}$ with Ind$(\widehat{A})=m$. Then $\widehat{A}$ can be expressed as  $\widehat{A}=\widehat{T}+\widehat{N}$, where
    \begin{enumerate}[\rm(i)]
        \item Ind$(\widehat{T})=1$,
        \item $\widehat{N}^{m}=O$,
        \item $\widehat{T}^{T}\widehat{N}=\widehat{N}\widehat{T}=O$.
    \end{enumerate}
Here $\widehat{T}$ is core partial and $\widehat{N}$ is nilpotent partial. Moreover, the decomposition is unique. 
\end{theorem}
\begin{proof}
     Let rank$\left(A^r\right)=t$, Ind$(\widehat{A})=m$, and $A\text{Ind}(\widehat{A})=r$.
Using Theorem \ref{decompprev}, for a unitary matrix $\widehat{U}=U+\epsilon U_0 \in \mathbb{D}\mathbb{R}^{n \times n}$ such that $\widehat{A}$ is of the form \eqref{Ahatexprsn}. Further, we can write
$$
\widehat{T}=\widehat{U}\left[\begin{array}{ll}
\widehat{T_{1}} & \widehat{T_{2}} \\
O & O
\end{array}\right] \widehat{U}^{T} \text { and } \widehat{N}=\widehat{U}\left[\begin{array}{ll}
O & O \\
O & \widehat{N}
\end{array}\right] \widehat{U}^{T}.
$$
Thus $\widehat{A}=\widehat{T}+\widehat{N}$, Ind $(\widehat{T})=1, \widehat{N}^m=O$ and $\widehat{T}^{T}\widehat{N}=\widehat{N}\widehat{T}=O$.
Next, we show that the decomposition is unique.
Let $\widehat{A}=\widehat{T}+\widehat{N}$ be a dual core-EP decomposition, then $\widehat{T}$ and $\widehat{N}$ satisfy conditions (i), (ii) and (iii). Write $\widehat{T}=T+\epsilon T_0$ and $\widehat{N}=N+\epsilon N_0$, where $T, T_0, N$ and $N_0 \in \mathbb{R}^{n \times n}$. From three conditions (i), (ii) and (iii), it is clear that $T$ is the core part and $N$ the nilpotent part of $A$, respectively. Suppose that $A, T$ and $N$ be of the form \eqref{eqncEPD} and $U^{T}BU$ be partitioned as \eqref{Bdcomp}. Since Ind$(\widehat{T})=1$, then $U^{T}T_0U$ can be partitioned as
\begin{equation}\label{eqn4.7}
    U^{T}T_0U=\left[\begin{array}{cc}
T_{11} & T_{12} \\
T_{13} & O
\end{array}\right],
\end{equation}
with $T_{11} \in \mathbb{R}^{t \times t}$. Partition $U^{T}N_0U$ as
\begin{equation}\label{eqn4.8}
U^{T}N_0U=\left[\begin{array}{ll}
N_{11} & N_{12} \\
N_{13} & N_{14}
\end{array}\right],
\end{equation}
with $N_{11} \in \mathbb{R}^{t \times t}$. Since $T_0+N_0=B$, then
\begin{equation}\label{eqn4.9}
    T_{12}+N_{12}=B_2,
\end{equation}

\begin{equation}\label{eqn4.10}
   N_{14}=B_4. 
\end{equation}
As $T_1$ is invertible, using \eqref{eqncEPD}, \eqref{eqn4.7}, \eqref{eqn4.8}, $\widehat{T}^{T} \widehat{N}=O$ and
$$
\begin{aligned}
\widehat{T}^{T} \widehat{N} & =U\left(\left[\begin{array}{cc}
T_1^{T} & O \\
T_2^{T} & O
\end{array}\right]+\epsilon\left[\begin{array}{cc}
T_{11}^{T} & T_{13}^{T} \\
T_{12}^{T} & O
\end{array}\right]\right) U^{T} U\left(\left[\begin{array}{cc}
O & O \\
O & N
\end{array}\right]+\epsilon\left[\begin{array}{ll}
N_{11} & N_{12} \\
N_{13} & N_{14}
\end{array}\right]\right) U^{T} \\
& =\epsilon U\left[\begin{array}{cc}
T_1^{T} N_{11} & T_1^{T}N_{12}+T_{13}^{T}N \\
T_2^{T} N_{11} & T_2^{T} N_{12}
\end{array}\right] U^{T}
\end{aligned},
$$
we obtain
\begin{equation}\label{eqn4.11}
N_{11}=O,~T_2^{T} N_{12}=O,~T_1^{T}N_{12}+T_{13}^{T}N=O.
\end{equation} 
If $\widehat{A}=\widehat{T}^{\prime}+\widehat{N}^{\prime}$ is another  dual core-EP decomposition, then by the uniqueness of the core-EP decomposition $A=T+N$ of $A$ is unique, we have  $\widehat{T}^{\prime}=T+\epsilon T_0^{\prime}$ and $\widehat{N}^{\prime}=N+\epsilon N_0^{\prime}$, with $T_0^{\prime}$ and $N_0^{\prime} \in \mathbb{R}^{n \times n}$. If we partition $U^{T} T_0^{\prime} U$ and $U^{T} N_0^{\prime} U$ as
$$
U^{T} T_0^{\prime} U=\left[\begin{array}{cc}
T_{11}^{\prime} & T_{12}^{\prime} \\
T_{13}^{\prime} & O
\end{array}\right], U^{T} N_0^{\prime} U=\left[\begin{array}{cc}
N_{11}^{\prime} & N_{12}^{\prime} \\
N_{13}^{\prime} & N_{14}^{\prime}
\end{array}\right],
$$
with $T_{11}^{\prime}$ and $N_{11}^{\prime} \in \mathbb{R}^{r \times r}$, then
 \begin{equation}\label{eqn4.12}
T_{12}^{\prime}+N_{12}^{\prime}=B_2,~N_{14}^{\prime}=B_4,
\end{equation} and
\begin{equation}\label{eqn4.13}
N_{11}^{\prime}=O,~T_2^{T} N_{12}^{\prime}=O,~T_1^{T}N_{12}^{\prime}+(T_{13}^{\prime})^{T}N=O.
\end{equation}
By using \eqref{eqn4.9}, \eqref{eqn4.10}, \eqref{eqn4.11}, \eqref{eqn4.12} and \eqref{eqn4.13}, we obtain
\begin{align}\label{align4.14}
\left\{\begin{array}{l}
\left(T_{12}-T_{12}^{\prime}\right)+\left(N_{12}-N_{12}^{\prime}\right)=O, \\
T_1^{T}\left(N_{12}-N_{12}^{\prime}\right)+\left((T_{13})^{T}-(T_{13}^{\prime})^{T}\right) N=O .
\end{array}\right.
\end{align}
Since $A\text{Ind}(\widehat{A})=r$, we get $N^r=O$. Post-multiplying  both sides $N^{r-1}$ of the second condition of \eqref{align4.14}, we obtain
$$
T_1^{T}\left(N_{12}-N_{12}^{\prime}\right) N^{r-1}+\left((T_{13})^{T}-(T_{13}^{\prime})^{T}\right) N^r=O.
$$

Thus $T_1^{T}\left(N_{12}-N_{12}^{\prime}\right) N^{r-1}=O$. Since $T_1$ is invertible, we have $\left(N_{12}-N_{12}^{\prime}\right) N^{r-1}=O$. Therefore, by Post-multiplying $N^{r-1}$ both sides in the second condition of \eqref{align4.14}, we obtain $\left(T_{12}-T_{12}^{\prime}\right) N^{r-1}=O$. Next, from the second condition of \eqref{align4.14}, we obtain
$$
T_1^{T}\left(N_{12}-N_{12}^{\prime}\right) N^{r-2}+\left((T_{13})^{T}-(T_{13}^{\prime})^{T}\right) N^{r-1}=O.
$$
Thus $\left(N_{12}-N_{12}^{\prime}\right) N^{r-2}=O$ and $\left((T_{13})^{T}-(T_{13}^{\prime})^{T}\right) N^{r-2}=O$. In similar manner, we can obtain $\left((T_{13})^{T}-(T_{13}^{\prime})^{T}\right)N=O$. Now it is from clear from second condition of \eqref{align4.14} that \begin{align}\label{eqn4.15}
    N_{12}=N_{12}^{\prime}.
\end{align}
Similarly, using condition $\widehat{N} \widehat{T}=O$, we obtain 
\begin{equation}\label{eqn4.16}
N_{13}=N_{13}^{\prime}.
\end{equation}
Now taking all equations together \eqref{eqn4.10}, \eqref{eqn4.11}, \eqref{eqn4.12}, \eqref{eqn4.15}, and \eqref{eqn4.16}, one can obtain $\widehat{N}=\widehat{N}^{\prime}$ which proves the uniqueness of dual core-EP decomposition.
\end{proof}
We can obtain the following theorem by analyzing the proof of Theorem \ref{DCEpdecomp}.

\begin{theorem} 
Consider $\widehat{A} \in \mathbb{D R}^{n \times n}$ and the dual core-EP decomposition of $\widehat{A}$ as given in the theorem \ref{DCEpdecomp}. Then there exists a unitary dual matrix $\widehat{U} \in \mathbb{D R}^{n \times n}$ such that
\begin{equation}\label{Dualcepdcmp}
\widehat{A}=\widehat{U}\left[\begin{array}{ll}
\widehat{T_1} & \widehat{T_2} \\
O & \widehat{N}
\end{array}\right] \widehat{U}^{T},
\end{equation}
where $\widehat{N} \in \mathbb{D R}^{(n-t) \times(n-t)}$ is dual nilpotent and $\widehat{T_1} \in \mathbb{D R}^{t \times t}$ is invertible. 
\end{theorem}

%The existence of the dual core generalized inverse is equivalent to the known fact that $\text{Ind}(\widehat{A})=1$. 
The dual Core generalized inverse $\widehat{A}^{\core}$ is a dual $\{1\}$-inverse. $\widehat{A}^{\ep}$ is not a dual $\{1\}$-inverse when $\text{Ind}(\widehat{A})\neq 1$. 
The product $\widehat{A} \widehat{A}^{\ep} \widehat{A}$ is important to discuss the dual CEP generalized inverse, even though in this case $\widehat{A} \widehat{A}^{\ep} \widehat{A} \neq \widehat{A}$.

\begin{definition} Consider $\widehat{A}=A+\epsilon B \in \mathbb{D R}^{n \times n}$ such that the DCEPGI $\widehat{A}^{\ep}$ exists, then 
$$
\widehat{T}_A=\widehat{A}\widehat{A}^{\ep} \widehat{A}
$$
is called dual core part of $\widehat{A}$, and
$$
\widehat{N}_A=\widehat{A}-\widehat{A}\widehat{A}^{\ep} \widehat{A}
$$
is called dual nilpotent part of $\widehat{A}$.
\end{definition}
\begin{theorem} 
Consider $\widehat{A}\in \mathbb{D R}^{n \times n}$ be of the form \eqref{Dualcepdcmp} with $\text{AInd}(\widehat{A})=m,~ \text{Arank} (\widehat{A}^k)=t$ and let the DCEPGI $\widehat{A}^{\ep}$ of $\widehat{A}$ exists. Then
\begin{equation}
\widehat{A}^{\ep}=\widehat{U}\begin{bmatrix}
\widehat{T_1}^{-1} & O \\
O & O
\end{bmatrix} \widehat{U}^{T}.
\end{equation}
\end{theorem}
\begin{proof}
Since $\widehat{A}^{\ep}$ exists, it follows that  $\widehat{N}^{m}=O$. From \eqref{Dualcepdcmp}, we obtain
    $$
\widehat{A}^m=\widehat{U}\begin{bmatrix}
\widehat{T_1}^m & \widehat{H} \\
O & O
\end{bmatrix} \widehat{U}^{T},
$$
where $\tilde{H}=\sum_{i=0}^{m-1} \widehat{T_{1}}^{i}\widehat{T_{2}}\widehat{N}^{m-i-1}$. 
Consider
$$
\widehat{X}=\widehat{U}\begin{bmatrix}
\widehat{T_1}^{-1} & O \\
O & O
\end{bmatrix} \widehat{U}^{T} .
$$
Now, we verify all three conditions of definition of dual CEP inverse.
$$
\widehat{A} \widehat{X}=\widehat{U}\begin{bmatrix}
\widehat{T_1} & \widehat{T_2} \\
O & \widehat{N}
\end{bmatrix}\begin{bmatrix}
\widehat{T_1}^{-1} & O \\
O & O
\end{bmatrix} \widehat{P}^{-1}=\widehat{U}\begin{bmatrix}
I & O \\
O & O
\end{bmatrix} \widehat{U}^{T}=(\widehat{A} \widehat{X})^{T}.
$$
Further,
$$
\widehat{A} \widehat{X}^2=\widehat{U}\begin{bmatrix}
I & O \\
O & O
\end{bmatrix}\begin{bmatrix}
\widehat{T_1}^{-1} & O \\
O & O
\end{bmatrix} \widehat{U}^{T}=\widehat{U}\begin{bmatrix}
\widehat{T_1}^{-1} & O \\
O & O
\end{bmatrix} \widehat{U}^{T}=\widehat{X},
$$
and
$$
\widehat{X}\widehat{A}^{m+1} =\widehat{U}\begin{bmatrix}
\widehat{T_1}^{-1} & O \\
O & O
\end{bmatrix}\begin{bmatrix}
\widehat{T_1}^{m+1} & \widehat{T_1}\widehat{H} \\
O & O
\end{bmatrix} \widehat{U}^{T}=\widehat{U}\begin{bmatrix}
\widehat{T_1}^{m} & \widehat{H} \\
O & O
\end{bmatrix} \widehat{U}^{T}=\widehat{A}^m.
$$
Therefore, $\widehat{X}=\widehat{A}^{\ep}$.
\end{proof}
\begin{corollary}
    Consider $\widehat{A}\in \mathbb{D R}^{n \times n}$. Then $\widehat{A}^{\core}$ exists if and only if there exists a dual unitary matrix $\widehat{U}$ and a dual invertible matrix $\widehat{T_1}$ such that 
    \begin{equation}
        \widehat{A}=\widehat{U}\begin{bmatrix}
\widehat{T_1} & \widehat{T_2} \\
O & O
\end{bmatrix}\widehat{U}^{T}.
    \end{equation} Moreover, 
    \begin{equation}
        \widehat{A}^{\core}=\widehat{U}\begin{bmatrix}
\widehat{T_1}^{-1} & O \\
O & O
\end{bmatrix}\widehat{U}^{T}.
    \end{equation}
\end{corollary}
For DCEPGI, the compact formula is given below.
\begin{theorem}\label{compactformula} 
Consider $\widehat{A}=A+\epsilon B \in \mathbb{D}\mathbb{R}^{n \times n}$ with Ind$(A)=m$. Assume that $\widehat{A}^{m}={A}^{m}+\epsilon S$, where $S:=\sum_{i=1}^{m} A^{m-i}BA^{i-1}$. If  both DCEPGI $\widehat{A}^{\ep}$ and DDGI $\widehat{A}^{D}$ exist, then 

\begin{equation*}
\begin{split}
 \widehat{A}^{\ep} & =\widehat{A}^{D} \widehat{A}^{m} (\widehat{A}^{m})^{\dagger} \\
& \qquad =A^{\ep}+\epsilon [-A^{\ep}S(A^{m})^{\dagger}+A^{D}A^{m}((A^{m})^{T}A^{m})^{\dagger}S^{T}(I-A^{m}(A^{m})^{\dagger})+A^{D}S(A^{m})^{\dagger} \\ & \qquad-A^{D}BA^{D}A^{m}(A^{m})^{\dagger}+\sum_{i=0}^{m-1} A^{\pi}A^{i}B(A^{D})^{i+2}A^{m}(A^{m})^{\dagger}],
\end{split}
\end{equation*}
where $A^{\pi}=I-AA^{D}$.
\end{theorem}

\begin{proof} According to \cite[Theorem 2.2]{dualDrazin}, if $\widehat{A}^{D}$ exists, then  $(\widehat{A}^{m})^{\dagger}$ exists. Denote $\widehat{X}=\widehat{A}^{D} \widehat{A}^{m} (\widehat{A}^{m})^{\dagger}$. Clearly 
$$
\begin{gathered}
\widehat{X} \widehat{A}^{m+1}=\widehat{A}^{D} \widehat{A}^{m} (\widehat{A}^{m})^{\dagger} \widehat{A}^{m+1}=\widehat{A}^{D} \widehat{A}^{m+1}=\widehat{A}^{m}, \\
\widehat{A} \widehat{X}^2=\widehat{A} \widehat{A}^{D} \widehat{A}^{m}(\widehat{A}^{m})^{\dagger}\widehat{A}^{D} \widehat{A}^{m} (\widehat{A}^{m})^{\dagger}=\widehat{A}^{m} (\widehat{A}^{m})^{\dagger}\widehat{A}^{m} \widehat{A}^{D} (\widehat{A}^{m})^{\dagger}=\widehat{A}^{D} \widehat{A}^{m} (\widehat{A}^{m})^{\dagger}=\widehat{X}.
\end{gathered}
$$

Also $(\widehat{A} \widehat{X})^{T}=\left(\widehat{A} \widehat{A}^{D} \widehat{A}^{m} (\widehat{A}^{m})^{\dagger}\right)^{T}=\left(\widehat{A}^{m} (\widehat{A}^{m})^{\dagger}\right)^{T}=\widehat{A}^{m} (\widehat{A}^{m})^{\dagger}=\widehat{A} \widehat{X}$.

Therefore,  $\widehat{A}^{\ep}=\widehat{X}=\widehat{A}^{D} \widehat{A}^{m} (\widehat{A}^{m})^{\dagger}$. Since $(\widehat{A}^{m})^{\dagger}$ exists then one can get $(\widehat{A}^{m})^{\dagger}=(A^{m})^{\dagger}+\epsilon R$, where
\begin{align*}
R=-(A^{m})^{\dagger}S(A^{m})^{\dagger}+((A^{m})^{T}A^{m})^{\dagger}S^{T}(I-A^{m}(A^{m})^{\dagger})+(I-(A^{m})^{\dagger}A^{m})S^{T}(A^{m}(A^{m})^{T})^{\dagger}.
\end{align*}
Substituting the value of $R$ and equation \eqref{A^D}  into $\widehat{A}^{D} \widehat{A}^{m} (\widehat{A}^{m})^{\dagger}$, we obtain the desired result.
\end{proof}
Additionally, using equations \eqref{Bdcomp}, \eqref{eqncEPD}, \eqref{cepdcmp}, \eqref{Sdcomp}, and  Theorem \ref{equivalent}, we have the following canonical representation of the dual CEP inverse of $\widehat{A}$.
\begin{theorem}
 For $\widehat{A}=A+\epsilon B \in \mathbb{D}\mathbb{R}^{n \times n}$, consider $A$ and $B$ be respectively, of the form \eqref{eqncEPD} and \eqref{Bdcomp},  respectively, and the DCEPGI of $\widehat{A}$ exists. Then, \begin{align}\label{dualCEPexprsn}
      \begin{split}  & \widehat{A}^{\ep}=U\begin{bmatrix}T_{1}^{-1} & O \\ O & O\end{bmatrix} U^{T} + \\ & \epsilon U\begin{bmatrix}
-T_{1}^{-1}B_{1}T_{1}^{-1}-T_{1}^{-1}T_{2}\sum_{i=1}^{m}N^{m-i}B_3T_{1}^{i-m-2} & T_{1}^{-1}(\sum_{i=1}^{m}N^{m-i}B_3T_{1}^{i-m-1})^{T}\\\sum_{i=1}^{m}N^{m-i}B_3T_{1}^{i-m-2} & O
\end{bmatrix}U^{T}.
   \end{split} 
   \end{align}
\end{theorem}

\section{Relations with other dual generalized inverses}\label{sec5}

In this section, by considering a few characterizations of DCEPGI  and 
%in the form of $\widehat{A}^{\ep}=A^{\ep}-\epsilon A^{\ep}BA^{\ep}$ and 
DDGI, 
%in the form of $\widehat{A}^{D}=A^{D}-\epsilon A^{D}BA^{D}$,
we investigate the relationships between different dual generalized inverses.
\begin{theorem}\label{relationthm}
    Consider $\widehat{A}=A+\epsilon B \in \mathbb{D}\mathbb{R}^{n \times n}$ and $\text{Ind}(A)=m$. Assume that $\widehat{A}^{m}={A}^{m}+\epsilon S$, where $S:=\sum_{i=1}^{m} A^{m-i}BA^{i-1}$. If the DCEPGI $\widehat{A}^{\ep}$ exists, then the following are equivalent: 
    \begin{enumerate}[\rm(i)]
        \item  $\widehat{A}^{\ep}=A^{\ep}-\epsilon A^{\ep}BA^{\ep}$,
        \item $(I-A^{m}(A^{m})^{\dagger})S=O$,
        \item  $(I-AA^{\ep})S=O$,
        \item $AA^{\ep}S=SAA^{\ep}=S$,
        \item $\rg({S})\subseteq \rg({A^{m}})$ and $\nl((A^{m})^{*}) \subseteq \nl({S})$.
    \end{enumerate}
\end{theorem}
\begin{proof}
(i)$\Rightarrow$(ii):  Let $A$ and $B$ be of the form \eqref{eqncEPD} and \eqref{Bdcomp}, respectively, and let $A$ be core-EP invertible. Then,
\begin{equation}\label{eqn6.1}
    -A^{\ep}BA^{\ep}=U\begin{bmatrix}
-T_{1}^{-1}B_{1}T_{1}^{-1} & O \\
O & O\end{bmatrix}U^{T},
\end{equation}
\begin{equation}
    -A^{\ep}S(A^{m})^{\dagger}+A^{D}S(A^{m})^{\dagger}=O,
\end{equation}
\begin{equation}
    A^{D}A^{m}((A^{m})^{T}A^{m})^{\dagger}S^{T}(I-A^{m}(A^{m})^{\dagger})=U\begin{bmatrix}
O & T_{1}^{-1}(\sum_{i=1}^{m} N^{m-i}B_{3}T_{1}^{i-m-1})^{T} \\
O & O\end{bmatrix}U^{T},
\end{equation}
\begin{equation}
    -A^{D}BA^{D}A^{m}(A^{m})^{\dagger}=U\begin{bmatrix}
-T_{1}^{-1}B_{1}T_{1}^{-1} & O \\
O & O\end{bmatrix}U^{T},
\end{equation}
\begin{equation}\label{eqn6.5}
    \sum_{i=0}^{m-1} A^{\pi}A^{i}B(A^{D})^{i+2}A^{m}(A^{m})^{\dagger}=U\begin{bmatrix}
-T_{1}^{-1}T_{2}\sum_{i=1}^{m}N^{m-i}B_3T_{1}^{i-m-2} & O \\
\sum_{i=1}^{m} N^{m-i}B_{3}T_{1}^{i-m-2} & O\end{bmatrix}U^{T}, 
\end{equation} and 
\begin{equation}\label{eqn6.6}
    (I-A^{m}(A^{m})^{\dagger})S=U\begin{bmatrix}
O & O \\
\sum_{i=1}^{m} N^{m-i}B_{3}T_{1}^{i-1} & \sum_{i=1}^{m} N^{m-i}B_3F+N^{m-i}B_4 N^{i-1} \end{bmatrix}U^{T}.
\end{equation}
 Since the DCEPGI $\widehat{A}^{\ep}$ of $\widehat{A}=A+\epsilon B \in \mathbb{D}\mathbb{R}^{n \times n}$ exists and $\widehat{A}^{\ep}=A^{\ep}-\epsilon A^{\ep}BA^{\ep}$. By applying Theorem \ref{equivalent}, we obtain $\sum_{i=1}^{m} N^{m-i}B_3F+N^{m-i}B_4 N^{i-1}=\sum_{i=1}^{m} N^{m-i}B_{3}T_{1}^{i-m-1}\Tilde{T}$. To show $\widehat{A}^{\ep}=A^{\ep}-\epsilon A^{\ep}BA^{\ep}$, from Theorem \ref{compactformula}, we obtain
\begin{center}
$A^{\ep}BA^{\ep}=-A^{\ep}S(A^{m})^{\dagger}+A^{D}A^{m}((A^{m})^{T}A^{m})^{\dagger}S^{T}(I-A^{m}(A^{m})^{\dagger})+A^{D}S(A^{m})^{\dagger}-A^{D}BA^{D}A^{m}(A^{m})^{\dagger}+\sum_{i=0}^{m-1} A^{\pi}A^{i}B(A^{D})^{i+2}A^{m}(A^{m})^{\dagger}.$
\end{center}
From \eqref{eqn6.1}-\eqref{eqn6.5}, it follows that $\sum_{i=1}^{m} N^{m-i}B_{3}T_{1}^{i-1}=O$ and $\sum_{i=1}^{m} N^{m-i}B_{3}T_{1}^{i-m-1}\Tilde{T}=O$. Hence from the equation \eqref{eqn6.6}, we obtain $(I-A^{m}(A^{m})^{\dagger})S=O$.\\
(ii)$\Rightarrow$(i):  Suppose that $(I-A^{m}(A^{m})^{\dagger})S=O$. Since the DCEPGI of $\widehat{A}$ exists, then $A^{\ep}$ is the real part of the DCEPGI, i.e., $\widehat{A}^{\ep}=A^{\ep}+\epsilon R$. From $(I-A^{m}(A^{m})^{\dagger})S=O$, we obtain $\sum_{i=1}^{m} N^{m-i}B_{3}T_{1}^{i-1}=O$. Now the expression from Theorem \ref{compactformula}, for $A^{\pi}=I-AA^{D}$, we obtain 
$$
\begin{aligned}
R & =-A^{\ep}S(A^{m})^{\dagger}+A^{D}A^{m}((A^{m})^{T}A^{m})^{\dagger}S^{T}(I-A^{m}(A^{m})^{\dagger})+A^{D}S(A^{m})^{\dagger}\\ & \qquad -A^{D}BA^{D}A^{m}(A^{m})^{\dagger}+\sum_{i=0}^{m-1} A^{\pi}A^{i}B(A^{D})^{i+2}A^{m}(A^{m})^{\dagger} \\
& \qquad =U\begin{bmatrix}
-T_{1}^{-1}B_1T_{1}^{-1} & O \\
O & O
\end{bmatrix} U^{T}=-A^{\ep}BA^{\ep}.
\end{aligned}
$$
Therefore, $\widehat{A}^{\ep}=A^{\ep}-\epsilon A^{\ep} B A^{\ep}$.\\
(i)$\Leftrightarrow$(iii): Similar to (i)$\Leftrightarrow$(ii).\\
(i)$\Leftrightarrow$(iii): It follows from (i)$\Leftrightarrow$(ii).\\
(iv)$\Leftrightarrow$(v): Since $AA^{\ep}=P_{\rg(A^{m}),\nl((A^{m})^{*})}$ is a projector, then $AA^{\ep}S=S$ if and only if $\rg({S})\subseteq \rg({A^{m}})$ and also $SAA^{\ep}=S$ if and only if $\nl((A^{m})^{*}) \subseteq \nl({S})$.
\end{proof}
Utilizing the equivalence property between $(I-A^{m}(A^{m})^{\dagger})S=0$ and $\text{rank}\left(\begin{bmatrix}A^{m} & S\end{bmatrix}\right)=\text{rank}(A^{m})$, we obtain the following result.
\begin{theorem}
    Consider $\widehat{A}=A+\epsilon B \in \mathbb{D}\mathbb{R}^{n \times n}$ with $\text{Ind}(A)=m$ and the DCEPGI $\widehat{A}^{\ep}$ exists. Denote $\widehat{A}^{m}={A}^{m}+\epsilon S$, where $S:=\sum_{i=1}^{m} A^{m-i}BA^{i-1}$. Then
    \begin{center}
      $
\text{rank}\left(\begin{bmatrix}
    A^{m} & S
\end{bmatrix}\right)=\text{rank}(A^{m})\iff  \widehat{A}^{\ep}=A^{\ep}-\epsilon A^{\ep}BA^{\ep}$.  
    \end{center}
    
\end{theorem}
\begin{theorem} \label{thm6.3}
Consider $\widehat{A}=A+\epsilon B \in \mathbb{D}\mathbb{R}^{n \times n}$ with $\text{Ind}(A)=m$. Denote $\widehat{A}^{m}={A}^{m}+\epsilon S$, where $S:=\sum_{i=1}^{m} A^{m-i}BA^{i-1}$. If $~(\widehat{A}^{m})^{\dagger}$ exists and $(\widehat{A}^{m})^{\dagger}=(A^{m})^{\dagger}-\epsilon (A^{m})^{\dagger}S(A^{m})^{\dagger}$, then the DCEPGI exists and $\widehat{A}^{\ep}=A^{\ep}-\epsilon A^{\ep}BA^{\ep}$.
\end{theorem}
\begin{proof}
Given that $\text{Ind}(A)=m$ and $(\widehat{A}^{m})^{\dagger}=(A^{m})^{\dagger}-\epsilon (A^{m})^{\dagger}S(A^{m})^{\dagger}$, then using Lemma \ref{lem2.3}, $\left(I-A^{m}(A^{m})^{\dagger}\right)S=0$. In view of Theorem \ref{relationthm}, we get  $\left(I-A^{m}(A^{m})^{\dagger}\right)S=0$ and $\text{Ind}(A)=m$. Which implies  the existence of DCEPGI of $\widehat{A}$, and verifies the result $\widehat{A}^{\ep}=A^{\ep}-\epsilon A^{\ep}BA^{\ep}$.
\end{proof}
\begin{theorem} 
Consider $\widehat{A}=A+\epsilon B \in \mathbb{D}\mathbb{R}^{n \times n}$. Denote $\widehat{A}^{m}={A}^{m}+\epsilon S$, where $S:=\sum_{i=1}^{m} A^{m-i}BA^{i-1}$. If $\widehat{A}$ has a DDGI and $\widehat{A}^{D}=A^{D}-\epsilon A^{D}BA^{D}$, then DCEPGI exists and  $\widehat{A}^{\ep}=A^{\ep}-\epsilon A^{\ep}BA^{\ep}$.
\end{theorem}

\begin{proof} 
Since $\widehat{A}^{D}$ exists, and $\widehat{A}^{D}=A^{D}-\epsilon A^{D}BA^{D}$, then from Theorem \ref{thm6.3} $\text{Ind}(A)=m$, $~(\widehat{A}^{m})^{\dagger}$ exists and $(\widehat{A}^{m})^{\dagger}=(A^{m})^{\dagger}-\epsilon (A^{m})^{\dagger}S(A^{m})^{\dagger}$. Therefore, the DCEPGI of $\widehat{A}$ exists, and completes the proof of $\widehat{A}^{\ep}=A^{\ep}-\epsilon A^{\ep}BA^{\ep}$.
\end{proof}

Next, we present certain characteristics of dual CEP generalized inverse, that are analogous to those of square real matrices. The dual matrix $\widehat{A}=A+\epsilon B \in \mathbb{D}\mathbb{R}^{n \times n}$ has the following range and null spaces:

\begin{align}\label{6.7}
\rg(\widehat{A})=\{\widehat{y} \in \mathbb{D}\mathbb{R}^n : \widehat{y}=\widehat{A}\widehat{x}, \widehat{x} \in \mathbb{D}\mathbb{R}^n\}=\left\{A z+\epsilon(Aw+Bz): z, w \in \mathbb{R}^n\right\}, \\
 \label{6.8} \nl(\widehat{A})=\{\widehat{x} \in \mathbb{D}\mathbb{R}^n: \widehat{A}\widehat{x}=O\}=\left\{y+\epsilon z: Ay=O, Az+By=O: y, z \in \mathbb{R}^n\right\} .
\end{align}

\begin{theorem}\label{rangenull}
    Consider $\widehat{A}=A+\epsilon B \in \mathbb{D}\mathbb{R}^{n \times n}$ with $\text{Ind}(A)=m$. Assume that $\widehat{A}^{m}={A}^{m}+\epsilon S$, where $S:=\sum_{i=1}^{m} A^{m-i}BA^{i-1}$. If $\widehat{A}$ has a DCEPGI and $\widehat{A}^{\ep}=A^{\ep}-\epsilon A^{\ep}BA^{\ep}$, then  
    \begin{enumerate}[\rm(i)]
        \item $\rg(\widehat{A}^{\ep})=\rg(\widehat{A}^{m})$, \item $\nl(\widehat{A}^{\ep})=\nl((\widehat{A}^{m})^{*})$,
        \item $\rg(\widehat{A}^{m})\cap \nl((\widehat{A}^{m})^{*})=\{O\}$.
    \end{enumerate}
\end{theorem}
\begin{proof}
    (i) Since DCEPGI of $\widehat{A}$ exists, by definition, $\widehat{A}^{m}=\widehat{A}^{\ep}\widehat{A}^{m+1}$. Hence, $\rg(\widehat{A}^{m})\subseteq \rg(\widehat{A}^{\ep})$. Again, since $\widehat{A}^{\ep}=\widehat{A}(\widehat{A}^{\ep})^{2}=(\widehat{A})^{m}(\widehat{A}^{\ep})^{m}\widehat{A}^{\ep}$, $\rg(\widehat{A}^{\ep})\subseteq \rg(\widehat{A}^{m})$. Therefore, $\rg(\widehat{A}^{\ep})=\rg(\widehat{A}^{m})$.

    (ii) Since the dual CEP inverse of $\widehat{A}$ exists, it can be seen that  $\widehat{A}^{\ep}=\widehat{A}^{\ep}\widehat{A}\widehat{A}^{\ep}=\widehat{A}^{\ep}\widehat{A}^{m}(\widehat{A}^{\ep})^{m}=\widehat{A}^{\ep}(\widehat{A}^{m}(\widehat{A}^{\ep})^{m})^{*}=\widehat{A}^{\ep}((\widehat{A}^{\ep})^{m})^{*}(\widehat{A}^{m})^{*}$, Thus $\nl((\widehat{A}^{m})^{*}) \subseteq \nl(\widehat{A}^{\ep})$. It follows from $\widehat{A}^{\ep}=\widehat{A}^{D} \widehat{A}^{m}(\widehat{A}^{m})^{\dagger}$ that $\nl(\widehat{A}^{\ep}) =\nl(\widehat{A}^{D} \widehat{A}^{m}(\widehat{A}^{m})^{\dagger}) \subseteq \nl(\widehat{A}^{m}(\widehat{A}^{m})^{\dagger})\subseteq \nl((\widehat{A}^{m})^{\dagger})=\nl((\widehat{A}^{m})^{*})$. Hence $\nl(\widehat{A}^{\ep})=\nl((\widehat{A}^{m})^{*})$.
     
    (iii) For any $\widehat{x} \in \rg(\widehat{A}^{m}) \cap \nl((\widehat{A}^{m})^{*})$, using \eqref{6.7} and \eqref{6.8}, there exist $y, z \in \mathbb{R}^n$ such that $\widehat{x}=\widehat{A}^{m}(y+\epsilon z)=A^{m} y+\epsilon(A^{m} z+Sy)$, and
    $$
    [(A^{m})^{*}+\epsilon S^{*}][A^{m} y+\epsilon(A^{m} z+Sy)]=(A^{m})^{*}A^{m}y+\epsilon[(A^{m})^{*}A^{m}z+((A^{m})^{*}S+S^{*}A^{m})y]=O.
    $$

    Hence $(A^{m})^{*}A^{m}y=O$ and $(A^{m})^{*}A^{m}z+((A^{m})^{*}S+S^{*}A^{m})y=O$.
    Since it is known that $A^{\ep}=P_{\rg(A^{m}),\nl((A^{m})^{*})}$, one can observe from $(A^{m})^{*}A^{m}y=0$ that $A^{m}y \in \rg(A^{m}) \cap \nl((A^{m})^{*})=\{O\}$. This, $A^{m}y=O$. Therefore, $O=(A^{m})^{*}A^{m}z+((A^{m})^{*}S+S^{*}A^{m})y=(A^{m})^{*}A^{m}z+(A^{m})^{*}Sy=(A^{m})^{*}[A^{m}z+Sy]$, i.e., $A^{m}z+Sy \in \nl((A^{m})^{*})$. Since $\widehat{A}^{\ep}$ exists, then from (iii) of Theorem \ref{rangenull}, we can write $\left(I-A A^{\ep}\right)S\left(I-A A^{\ep}\right)=O$, i.e.,
  \begin{equation}\label{6.9}
    S=AA^{\ep}S+SAA^{\ep}-AA^{\ep}SAA^{\ep}.
   \end{equation}

    Applying \eqref{6.9} into $A^{m}z+Sy$, we obtain
    $$
    \begin{aligned}
    A^{m}z+Sy & =A^{m}z+(AA^{\ep}S+SAA^{\ep}-AA^{\ep}SAA^{\ep}) y=A^{m}z+AA^{\ep}Sy \\
    & =A^{m}z+A^{m}(A^{\ep})^{m}Sy=A^{m}[z+(A^{\ep})^{m}Sy] \in \rg(A^{m}) .
    \end{aligned}
    $$

    Hence, $A^{m}z+Sy \in \rg(A^{m}) \cap \nl((A^{m})^{*})=\{O\}$, i.e., $A^{m}z+Sy=O$. Therefore, $\widehat{x}=A^{m} y+\epsilon(A^{m} z+Sy)=O$, which proves that $\rg(\widehat{A}^{m})\cap \nl((\widehat{A}^{m})^{*})=\{O\}$.
\end{proof}
We now present the reverse-order and forward-order laws for DCEPGI. We will first illustrate that neither reverse-order nor forward-order laws hold for the DCEPGI.
\begin{example}\rm
    Let $\widehat{C}=\begin{bmatrix}
        2 & 1 & 0 \\ -1 & 3 & 0 \\ -2 & 1 & 0
    \end{bmatrix}+\epsilon \begin{bmatrix}
        2 & 2 & 4 \\ 3 & -1 & 2 \\ -4 & -2 & -6
    \end{bmatrix}$ and $\widehat{D}=\begin{bmatrix}
        1 & -1 & 0 \\ 0 & 2 & 0 \\ -1 & 3 & 0
    \end{bmatrix}+ \epsilon \begin{bmatrix}
        3 & -4 & 3 \\ 1 & 0 & -1 \\ 1 & -5 & 4
    \end{bmatrix}$. Then, we get $\widehat{C}^{\ep}=\begin{bmatrix}
        \frac{1}{2} & \frac{1}{4} & 0 \\ 0 & 0 & 0 \\ \frac{-1}{2} & \frac{-1}{4} & 0
    \end{bmatrix}+\epsilon\begin{bmatrix}
        \frac{3}{8} & \frac{3}{16} & 0 \\ 0 & 0 & 0 \\ \frac{-3}{8} & \frac{-3}{16} & 0
    \end{bmatrix},~\widehat{D}^{\ep}=\begin{bmatrix}
        0 & 3 & -1 \\ 0 & 0 & 0 \\ -1 & 6 & -1
    \end{bmatrix}+\epsilon\begin{bmatrix}
        -7 & 36 & -5 \\ 0 & 0 & 0 \\ -13 & 12 & 11
    \end{bmatrix}$, and $\widehat{C}\widehat{D}=\begin{bmatrix}
        2 & 0 & -2 \\ 0 & 0 & 0 \\ -2 & 0 & 2
    \end{bmatrix}+\epsilon\begin{bmatrix}
        5 & 4 & 3 \\ 1 & 6 & -3 \\ -3 & -10 & -3
    \end{bmatrix}$. Therefore, $(\widehat{C}\widehat{D})^{\ep}=\begin{bmatrix}
         \frac{1}{8} & 0 & \frac{-1}{8} \\ 0 & 0 & 0 \\ \frac{-1}{8} & 0 & \frac{1}{8}
    \end{bmatrix}+\epsilon\begin{bmatrix}
        \frac{-1}{32} & 0 & \frac{1}{32} \\ 0 & 0 & 0 \\ \frac{1}{8} & 0 & \frac{-1}{32}
    \end{bmatrix}$, $\widehat{C}^{\ep}\widehat{D}^{\ep}=\begin{bmatrix}
        0 & \frac{3}{2} & \frac{-1}{2} \\ 0 & 0 & 0 \\ 0 & \frac{-3}{2} & \frac{1}{2}
    \end{bmatrix}+\epsilon\begin{bmatrix}
        \frac{-7}{2} & \frac{153}{8}& \frac{-23}{8} \\ 0 & 0 & 0 \\ \frac{7}{2} & \frac{-153}{8}& \frac{23}{8}
    \end{bmatrix}$, and $\widehat{D}^{\ep}\widehat{C}^{\ep}=\begin{bmatrix}
        \frac{1}{2} & \frac{1}{4} & 0 \\ 0 & 0 & 0 \\ 0 & 0 & 0
    \end{bmatrix}+\epsilon\begin{bmatrix}
        \frac{-5}{8} & \frac{-5}{16} & 0 \\ 0 & 0 & 0 \\ -1 & \frac{-1}{2} & 0
    \end{bmatrix}$. Hence it is clear that $(\widehat{C}\widehat{D})^{\ep} \neq \widehat{C}^{\ep}\widehat{D}^{\ep}\neq \widehat{D}^{\ep}\widehat{C}^{\ep}$.
\end{example} 
Next, we provide sufficient conditions for both the reverse and forward-order laws for DCEPGI.
\begin{theorem}
   Consider $\widehat{A}=A+\epsilon A_{0}$ and $\widehat{B}=B+\epsilon B_{0}$ such that the $DCEPGI$ of $\widehat{A}, \widehat{B}, \widehat{A}\widehat{B}$ exist. If $AB=BA,~AB^{*}=B^{*}A,~B^{\ep}A_{0}=A_{0}B^{\ep}$, and $A^{\ep}B_{0}=B_{0}A^{\ep}$, then $(\widehat{A} \widehat{B})^{\ep}=\widehat{B}^{\ep} \widehat{A}^{\ep}=\widehat{A}^{\ep} \widehat{B}^{\ep}$.
\end{theorem} 

\begin{proof}
    The DCEPGI of $\widehat{A}$ and $\widehat{B}$ are $\widehat{A}^{\ep}=A^{\ep}-\epsilon A^{\ep}A_{0}A^{\ep}$ and $\widehat{B}^{\ep}=B^{\ep}-\epsilon B^{\ep}B_{0}C^{\ep}$, respectively. So, $\widehat{A}^{\ep} \widehat{B}^{\ep}=A^{\ep} B^{\ep}-\epsilon(A^{\ep} B^{\ep}B_{0} B^{\ep}+A^{\ep} A_{0} A^{\ep} B^{\ep})$. Since $\widehat{A} \widehat{B}=(A+\epsilon A_{0})(B+\epsilon B_{0})=AB+\epsilon(AB_{0}+A_{0}B)$, so we get  $(\widehat{A} \widehat{B})^{\ep}=$ $(AB)^{\ep}-\epsilon\left((AB)^{\ep}(AB_{0}+A_{0}B)(AB)^{\ep}\right)$. Both the equalities $AB=BA,~A^{*}B=BA^{*}$ imply that $(AB)^{\ep}=A^{\ep}B^{\ep}=B^{\ep}A^{\ep}$. Furthermore, \begin{center}$(\widehat{A} \widehat{B})^{\ep}=(AB)^{\ep}-\epsilon\left((AB)^{\ep}(AB_{0}+A_{0}B)(AB)^{\ep}\right)=A^{\ep}B^{\ep}-\epsilon\left(A^{\ep}B^{\ep}(AB_{0}+A_{0}B)A^{\ep}B^{\ep}\right)=A^{\ep}B^{\ep}-\epsilon(A^{\ep}B^{\ep}AB_{0}A^{\ep}B^{\ep}+A^{\ep}B^{\ep}A_{0}BA^{\ep}B^{\ep})=A^{\ep}B^{\ep}-\epsilon(B^{\ep}A^{\ep}AA^{\ep}B_{0}B^{\ep}+A^{\ep}A_{0}B^{\ep}BB^{\ep}A^{\ep})=A^{\ep} B^{\ep}-\epsilon(A^{\ep} B^{\ep}B_{0} B^{\ep}+A^{\ep} A_{0} A^{\ep} B^{\ep})=\widehat{A}^{\ep} \widehat{B}^{\ep}$. 
    \end{center} Similarly, we can show that $(\widehat{A} \widehat{B})^{\ep}=\widehat{B}^{\ep} \widehat{A}^{\ep}$.
\end{proof} 
For $\widehat{A} \in \mathbb{D}\mathbb{R}^{n \times n}$, $\widehat{x} \in \mathbb{D}\mathbb{R}^{n}$, and $\widehat{b} \in \mathbb{D}\mathbb{R}^{n}$, the dual linear equation is written as \begin{equation}\label{lineqn}
    \widehat{A}\widehat{x}=\widehat{b}.
\end{equation} 
\begin{theorem}
Consider $\widehat{A}=A+ \epsilon B \in \mathbb{D}\mathbb{R}^{n \times n}$ with $\text{Ind}(A)=m$. If  both DCEPGI $\widehat{A}^{\ep}$ and DDGI $\widehat{A}^{D}$ exist, then the system 
\begin{equation}\label{eq7.2}
    \widehat{A}^{m+1}\widehat{x}= \widehat{A}^{2m} (\widehat{A}^{m})^{\dagger}\widehat{b}
\end{equation}
is consistent and the general solution is given by 
\[\widehat{x}= \widehat{A}^{\ep}\widehat{b}+(I-\widehat{A}^{D}\widehat{A})\widehat{y},\]
where $\widehat{y} \in \mathbb{D}\mathbb{R}^{n}$ is any dual vector.

\end{theorem} 

\begin{proof}
Let $\widehat{y} \in \mathbb{D}\mathbb{R}^{n}$ and $\widehat{x}=\widehat{A}^{\ep}\widehat{b}+(I-\widehat{A}^{D}\widehat{A})\widehat{y}$. Then by applying $\widehat{A}^{\ep} =\widehat{A}^{D} \widehat{A}^{m} (\widehat{A}^{m})^{\dagger}$, we obtain
\begin{center}
$\widehat{A}^{m+1}\widehat{x}=\widehat{A}^{m+1}\widehat{A}^{\ep}\widehat{b}+\widehat{A}^{m+1}(I-\widehat{A}^{D}\widehat{A})\widehat{y}=\widehat{A}^{m+1}\widehat{A}^{\ep}\widehat{b}=\widehat{A}^{m+1}\widehat{A}^{D} \widehat{A}^{m} (\widehat{A}^{m})^{\dagger}\widehat{b}=\widehat{A}^{2m} (\widehat{A}^{m})^{\dagger}\widehat{b}$.
\end{center}
Thus $\widehat{x}$ satisfies the equation \eqref{eq7.2}. If $\widehat{z}$ is any arbitrary  solution of\eqref{eq7.2}, then by using $\widehat{A}^{\ep} =\widehat{A}^{D} \widehat{A}^{m} (\widehat{A}^{m})^{\dagger}$, we obtain 
\[\widehat{A}^{\ep}\widehat{b}=\widehat{A}^{D} \widehat{A}^{m} (\widehat{A}^{m})^{\dagger}\widehat{b}=(\widehat{A}^{D})^{m+1}\widehat{A}^{2m} (\widehat{A}^{m})^{\dagger}\widehat{b}=(\widehat{A}^{D})^{m+1}\widehat{A}^{m+1} \widehat{z}=\widehat{A}^{D}\widehat{A}\widehat{z}.\]
Now 
\[\widehat{z}=\widehat{A}^{\ep}\widehat{b}+\widehat{z}-\widehat{A}^{\ep}\widehat{b}=\widehat{A}^{\ep}\widehat{b}+\widehat{z}-\widehat{A}^{D}\widehat{A}\widehat{z}=\widehat{A}^{\ep}\widehat{b}+(I-\widehat{A}^{D}\widehat{A})\widehat{z}.\]

\end{proof}
Next, we have the following uniqueness result using the concept of Theorem \ref{rangenull}.
\begin{theorem} 
Consider $\widehat{A}=A+ \epsilon B \in \mathbb{D}\mathbb{R}^{n \times n}$ with $\text{Ind}(A)=m$. Assume that $\widehat{A}^{\ep}=A^{\ep}-\epsilon A^{\ep}BA^{\ep}$. Then the unique solution  of
\begin{equation}\label{uniqlinear}
    \widehat{A} \widehat{A}^{\ep}\widehat{x}=\widehat{A}^{\ep} \widehat{b}
\end{equation}
 in the range $\rg(\widehat{A}^{m})$, is given by $\widehat{A}^{\ep}b$.
\end{theorem}
\begin{proof} 
From the existence of $\widehat{A}^{\ep}$, we clearly obtain  that $\widehat{A}^{\ep} \widehat{b}$ is a solution to the system \eqref{uniqlinear}.
By (i) of Theorem \ref{rangenull}, $\widehat{A}^{\ep} \widehat{b} \in \rg(\widehat{A}^{\ep})=\rg(\widehat{A}^{m})$. Consider $\widehat{w}$ be another solution of \eqref{uniqlinear} in $\rg(\widehat{A}^{m})$. Then $\widehat{w}-\widehat{A}^{\ep} \widehat{b} \in \rg(\widehat{A}^{m})$. Since $\widehat{w}$ and $\widehat{A}^{\ep} \widehat{b}$ are solutions to \eqref{uniqlinear}, so we have $\widehat{A}\widehat{A}^{\ep}(\widehat{w}-\widehat{A}^{\ep} \widehat{b})=\mathbf{0}$. It follows from $\nl(\widehat{A}\widehat{A}^{\ep})= \nl(\widehat{A}^{\ep})=\nl((\widehat{A}^{m})^{*})$ that  $\widehat{w}-\widehat{A}^{\ep} \widehat{b} \in \nl((\widehat{A}^{m})^{*})$. Therefore, by (iii) of Theorem \ref{rangenull}, $\widehat{w}-\widehat{A}^{\ep} \widehat{b} \in \rg(\widehat{A}^{m})\cap \nl((\widehat{A}^{m})^{*})=\{\mathbf{0}\}$, i.e., $\widehat{w}=\widehat{A}^{\ep} \widehat{b}$.
\end{proof}

\section{Conclusion}\label{sec9}

This study  explored the dual core-EP generalized inverse of dual real matrices. It demonstrates the existence and uniqueness of this inverse, proposes a unique dual decomposition, provides similar characterizations, a compact formula, and shows its application to linear dual equations.

Some potential future research directions related to our work include:
\begin{itemize}
    \item Developing iterative methods and perturbation bounds for dual generalized inverses.
    \item Investigating binary relations on different dual generalized inverses using this newly introduced dual decomposition.
    
\end{itemize}

\section*{Declarations}

\begin{itemize}
\item Funding: N\'estor Thome expresses his gratitude for the support received from the Universidad Nacional de La Pampa, Facultad de Ingenier\'ia (Grant Resol. Nro. 135/19), Universidad Nacional del Sur (Grant PGI 24/ZL22), and  Ministerio de Ciencia e Innovaci\'on  (Grant Red de Excelencia RED2022-134176-T), Spain.
\item Conflict of interest/Competing interests:  The authors have stated that no potential conflicts of interest could have influenced the outcomes or interpretations presented in this work. They affirm their commitment to transparency and integrity in their research.

\item Data availability: In the current study, it's important to note that no datasets were produced or subjected to analysis.

\end{itemize}

%%===========================================================================================%%
%% If you are submitting to one of the Nature Portfolio journals, using the eJP submission   %%
%% system, please include the references within the manuscript file itself. You may do this  %%
%% by copying the reference list from your .bbl file, paste it into the main manuscript .tex %%
%% file, and delete the associated \verb+\bibliography+ commands.                            %%
%%===========================================================================================%%

\bibliography{sn-bibliography}% common bib file
%% if required, the content of .bbl file can be included here once bbl is generated
%%\input sn-article.bbl

\end{document}